\documentclass[11pt]{article}
\usepackage[letterpaper,twoside,outer=1in,vmargin=1in,]{geometry}
\usepackage{setspace}

\usepackage{amsmath,amsfonts,amsthm,url,graphicx,subcaption,float,booktabs,bbm,enumitem,natbib}
\usepackage{algorithm,algpseudocode}
\usepackage{tikz,caption,forest}\usepackage{multicol}
\tikzset{
	treenode/.style = {shape=rectangle, rounded corners,
		draw, align=center, 
		minimum height=2ex, text depth=0.25ex,
		top color=white, bottom color=blue!20},
	root/.style     = {treenode, font=\Large\rmfamily, bottom color=red!30},
	env/.style      = {treenode, font=\ttfamily\normalsize},
}
\usetikzlibrary{arrows.meta}
\algnewcommand{\Initialize}[1]{%
	\State \textbf{Initialize:}
	\Statex \hspace*{\algorithmicindent}\parbox[t]{.8\linewidth}{\raggedright #1}
}
\algnewcommand{\algorithmicgoto}{\textbf{go to}}%
\algnewcommand{\Goto}[1]{\algorithmicgoto~step~\ref{#1}}%
\usepackage[normalem]{ulem}

\usepackage[hidelinks,bookmarksnumbered=true]{hyperref}
\usepackage{todonotes,bm}
\usepackage{tikz}
\usepackage{pgfplots}
\pgfplotsset{compat=newest}
\pgfdeclareplotmark{mystar}{
\node[star,star point ratio=2.25,minimum size=6pt,
inner sep=0pt,draw=black,solid,fill=red] {};
}
\usetikzlibrary{decorations.pathreplacing}
\usepackage{xcolor}

\usepgflibrary{arrows}
\usetikzlibrary{calc}
\usetikzlibrary{plotmarks}
\usetikzlibrary{arrows.meta,positioning}
\usetikzlibrary{patterns}

\usepackage{url}
\usepackage{longtable,tabularx,multirow,makecell}
\usepackage{amssymb}
\usepackage{psfrag}
\usepackage{amsmath}
\usepackage{setspace}
\newcommand{\exclude}[1]{}
\usepackage{bm}
\usepackage{mathtools}
\usepackage{bbm}
\usepackage[flushleft]{threeparttable}
\usepackage{algorithm}
\usepackage{algpseudocode}
\algdef{SE}[DOWHILE]{Do}{doWhile}{\algorithmicdo}[1]{\algorithmicwhile\ #1}%
\algnewcommand{\Or}{\textbf{or}}
\algnewcommand{\And}{\textbf{and}}

\usepackage{thmtools} 
\usepackage{thm-restate}

\usepackage{cleveref}
\usepackage{tikz}
\usetikzlibrary{calc}
\usetikzlibrary{plotmarks}
\usetikzlibrary{backgrounds}

\declaretheorem[name=Theorem]{theorem}
\declaretheorem[name=Proposition]{proposition}
\declaretheorem[name=Lemma]{lemma}

\declaretheorem[name=Definition]{definition}

\def\hat{\widehat}

\def \R{\mathcal{R}}

\DeclarePairedDelimiter\ceil{\lceil}{\rceil}
\DeclarePairedDelimiter\floor{\lfloor}{\rfloor}

\DeclareMathOperator{\conv}{conv}

\renewcommand*{\qed}{\hfill\ensuremath{\square}}

\usepackage{todonotes,bm}

\newtheorem{conj}[theorem]{Conjecture}
\newcommand{\T}{\mathsf{\scriptscriptstyle T}}
\usepackage{comment}

\usepackage{enumitem}
\usepackage{bigints}
\usepackage{mwe}

\renewcommand{\arraystretch}{1.5}

\let\cline\cmidrule

\allowdisplaybreaks

\newcommand{\ack}[1]{}

\renewcommand{\citep}[1]{\cite{#1}}
\renewcommand{\citet}[1]{\cite{#1}}

\author{Santanu S. Dey$^1$, Nan Jiang$^2$, Aleksandr M. Kazachkov$^3$, Andrea Lodi$^4$, Gonzalo Mu\~{n}oz$^5$\\
$^1$ Georgia Institute of Technology, Atlanta, GA, USA (santanu.dey@isye.gatech.edu)\\
$^2$ Hong Kong University of Science and Technology, Hong Kong, China  (nanjiang@ust.hk)\\
$^3$ University of Florida, Gainesville, FL, USA (akazachkov@ufl.edu)\\
$^4$ Jacobs Technion-Cornell Institute, Cornell Tech, New York, NY, USA  (al748@cornell.edu)\\
$^5$ Universidad de Chile, Santiago, Chile  (gonzalo.m@uchile.cl)
}
\title{Chvátal-Gomory Rounding of Eigenvector Inequalities for QCQPs}

\begin{document}
\maketitle
\begin{abstract}
   We introduce and analyze a class of valid inequalities for nonconvex quadratically constrained optimization problems (QCQPs)
which we call \emph{Eigen-CG inequalities}. 
These inequalities are obtained by applying a Chvátal--Gomory (CG) rounding to the well-known eigenvector inequalities for QCQPs, 
and transferring binary-valid inequalities to the continuous setting via a result of Burer and Letchford (2009).
We define three nested subfamilies and prove that they are strictly contained in one another. 
However, we show that the convex conic closure of two of these subfamilies is equal and, in fact, coincides with the
Boros--Hammer inequalities---a powerful family of inequalities that include, in particular, the triangle and McCormick inequalities.
Using this CG perspective, we also prove that dense Eigen-CG
inequalities are ineffective when used with the standard SDP+McCormick relaxation. 
This provides a complementary perspective on what is observed in practice: that sparse inequalities are impactful. 
Finally, based on these insights, we develop a computational strategy to find sparse Eigen-CG cuts and verify their effectiveness in nonconvex
QCQP instances. Our results confirm that density quickly degrades effectiveness, but that including sparse inequalities beyond triangle
inequalities can provide significant improvements in dual bounds.
\end{abstract}
\maketitle

\textbf{Key words.} Eigen-CG cuts, Boros--Hammer inequalities, QCQP

\section{Introduction}

We consider nonconvex quadratically constrained quadratic problems (QCQPs), where the feasible region is a compact set.
Without loss of generality, these can be modeled as

\begin{equation}
\label{QCQP}\tag{QCQP}
\begin{aligned}
\min_{x \in \R^n} \quad &x^\T Q_0 x + c_0^\T x \\
&x^\T Q_i x + c_i^\T x + d_i \leq 0 &\quad& i=1,\ldots, m\\
&x\in [0,1]^n.
\end{aligned}
\end{equation}

These are highly expressive models with wide applicability.
For example, they can represent any mixed-integer polynomial optimization problem with
a compact feasible region, which already captures a broad set of real-world problems.
Accompanying this expressiveness, there are important computational and theoretical challenges related to advancing QCQP solution techniques. 
Two major research thrusts are developing methods
for finding good solutions and for certifying their quality.

One successful approach for the latter is finding good bounds for QCQPs via \emph{cutting planes}, or \emph{cuts}; these can yield computationally tractable relaxations of QCQPs, which can also be iteratively refined.
Most common relaxations for \eqref{QCQP} can be deduced from the following reformulation, adding a symmetric matrix variable $X$:
\begin{equation}
\label{QCQP-Ext}
\begin{aligned}
\min_{x,X}
\quad &\langle X , Q_0 \rangle + c_0^\T x \\ 
&\langle X, Q_i \rangle + c_i^\T x + d_i \leq 0 &\qquad& i=1,\ldots, m \\ 
& x\in [0,1]^n \\
& X - xx^\T = 0,
\end{aligned}
\end{equation}
where $\langle A, B \rangle = \textup{trace}(A^{\top}B)$ is the standard inner product of matrices.
Here, all the non-convexity is absorbed in $X - xx^\T = 0$, and thus, a convex relaxation of this constraint yields a convex 
relaxation of the problem. 
One approach is a convex combination of all rank-1 matrices;
adopting the notation of Burer and Letchford~\cite{burer2009nonconvex} on \emph{\textbf{q}uadratic \textbf{p}rograms with \textbf{b}ox constraints}, we define:
\begin{equation}\label{eq:QPB}
\mbox{QPB}_n =\conv \{(x,X)\in [0,1]^{n+{n+1 \choose 2}}\,:\, X_{i,j} = x_i x_j, \, 1\leq i \leq  j \leq n\}.
\end{equation}
Many important families of valid inequalities for \eqref{QCQP-Ext} can be obtained from $\mbox{QPB}_n$.
This includes the McCormick inequalities \cite{mccormick1976computability}, which state that for every $i\neq j$,
\begin{equation}\label{eq:McCormick}
X_{i,j} \leq x_i,\quad X_{i,j} \leq x_j, \quad X_{i,j}\geq 0, \quad X_{i,j} \geq  x_i + x_j - 1, 
\end{equation}
and the triangle inequalities \cite{padberg1989boolean}, which state that for every triplet $i,j,k$ it holds that
\vspace{-\abovedisplayskip}
\begin{subequations}\label{eq:triangle}
\begin{align}
X_{i,j} &\ge X_{i,k} + X_{j,k} - x_k, \label{tri:1}\\
X_{i,k} &\ge X_{i,j} + X_{j,k} - x_j, \label{tri:2}\\
X_{j,k} &\ge X_{i,j} + X_{i,k} - x_i, \label{tri:3}\\
X_{i,j} + X_{i,k} + X_{j,k} &\ge x_i + x_j + x_k - 1. \label{tri:4}
\end{align}
\end{subequations}
These were originally developed for the binary case, but they are also valid for the continuous setting. We further expand on these and more inequalities below.\\

Padberg~\cite{padberg1989boolean} introduced the related set called the \emph{Boolean quadric polytope}:%
\footnote{Some authors refer to this as the correlation polytope~\cite{deza1997geometry, pitowsky1991correlation}.}
\begin{equation}\label{eq:BQP}
\mbox{BQP}_n = \conv\{(x,X)\in \{0,1\}^{n+{n\choose 2}}\,:\, X_{i,j} = x_i x_j, \, 1\leq i <  j \leq n\}.
\end{equation}
Note that $\mbox{BQP}_n$ does not include variables $X_{i,i}$, as $x_i^2 = x_i$ whenever $x$ is binary.
One remarkable, and perhaps counterintuitive, result by Burer and Letchford~\cite{burer2009nonconvex} states that valid inequalities for $\mathrm{BQP}_n$ are also valid for $\mathrm{QPB}_n$. %

\begin{proposition}[Corollary~1 in \cite{burer2009nonconvex}] \label{prop:burerletchford}
If a linear inequality
\[\sum_{1\leq i < j \leq n}\beta_{ij} X_{i,j} + \sum_{i=1}^n \alpha_i x_i \leq \gamma \]
is valid for $\mbox{BQP}_n$, then it is valid for $\mbox{QPB}_n$ as well.
\end{proposition}

This result is at the core of our construction.

An alternative perspective for obtaining valid inequalities follows from the relaxed condition $X-xx^\T \succeq 0$:  using Schur's complement, this is equivalent to enforcing
\begin{equation}
 \left[ \begin{array}{cc}
1 & x^\T \\
x & X
\end{array} 
\right]\succeq 0,
\label{SDP-constraint}
\end{equation}
and thus, for any $(v_0,v)\in \mathbb{R}^{n+1}$, the following is a valid inequality for $\mbox{QPB}_n$
\begin{equation*}
\left\langle \begin{bmatrix} v_0 \\ v\end{bmatrix}  \begin{bmatrix} v_0 \\ v\end{bmatrix}^\T , \left[ \begin{array}{cc}
1 & x^\T \\
x & X
\end{array} 
\right] \right\rangle  \geq 0,
\end{equation*}
which can be equivalently rewritten as
\begin{equation}
\label{eq:eigencut}
\sum_{1\leq i < j \leq n} 2v_iv_j X_{i,j} + \sum_{i=1}^n v_i^2 X_{i,i} + \sum_{i=1}^n 2v_i v_0 x_i + v_0^2 \geq 0.
\end{equation}
We refer to this family as \emph{eigen-cuts}, and they are guaranteed to remove any relaxation solution $(\tilde{x},\tilde{X})$ that violates \eqref{SDP-constraint}, as $(v_0,v)$ can be chosen as an eigenvector of a negative eigenvalue in such a case.
This family of cutting planes has been considered before (see \cite{ramana1994algorithmic,gruber2000onsemidefinite,SheFra02,QuaBelMar12,bienstock2020outer,baltean2019scoring}), and the main drawback has been repeatedly acknowledged: if $(v_0, v)$ are eigenvectors of $\left[ \begin{array}{cc}
1 & \tilde{x}^\T \\
\tilde{x} & \tilde{X}
\end{array} 
\right]$, then the resulting cut will typically be \emph{dense}, which can cause numerical instability when they are added in a cutting plane fashion. 
Some work has addressed this issue and focused on producing sparse eigen-cuts \cite{fukuda2001exploiting,QuaBelMar12,dey2022cutting}. 

\subsection{Contributions}
In this work, we observe that using Proposition \ref{prop:burerletchford} along with \eqref{eq:eigencut} and rounding \`{a} la \emph{Chv\'{a}tal-Gomory} (CG)~\cite{Chvatal73_cg-cuts},
one can derive the following family of
valid inequalities for $\mbox{QPB}_n$:
\begin{equation}
\label{eq:eigenCGcut}
\sum_{1\leq i < j \leq n} \lceil 2v_iv_j \rceil X_{i,j} + \sum_{i=1}^n \lceil v_i^2 + 2v_i v_0 \rceil x_i + \lfloor v_0^2 \rfloor \geq 0.
\end{equation}

We call these \emph{Eigen-CG inequalities} or \emph{Eigen-CG cuts}.
This family of inequalities generalizes known, and highly expressive, inequalities: the Boros and Hammer (BH) inequalities \cite{boros1993cut}.
The BH inequalities can be obtained by noting that, for any $(w_0,w)\in \mathbb{Z}^{n+1}$, the following holds: 
\[(w^\T x + w_0 - 1)(w^\T x + w_0) \ge 0 \quad \forall x\in \{0,1\}^n.\]
Expanding this product and replacing $x_i x_j$ by $X_{i,j}$, we obtain the BH inequalities
\begin{equation}
\label{eq:BHineq}
\sum_{1\leq i < j \leq n}  2w_iw_j  X_{i,j} + \sum_{i=1}^n  w_i(w_i + 2 w_0 - 1) x_i + w_0(w_0 - 1)  \geq 0.
\end{equation}

 The fact that BH inequalities \eqref{eq:BHineq} are a special case of \eqref{eq:eigenCGcut} is known~\cite{letchford2012binary}, but here we provide a derivation from a new CG-based perspective. 
Through this new lens, we also define two more families of inequalities $\mathcal{F}_1$ and $\mathcal{F}_2$ that are between the BH and the Eigen-CG inequalities.
These families are defined based on the ``gap'' left by the fact that Eigen-CG inequalities can be defined using \emph{any} $(v_0, v)\in \mathbb{R}^{n+1}$, while BH inequalities require $(w_0,w)\in \mathbb{Z}^{n+1}$.
We show strict containment between them, i.e., loosely speaking 
\[ \mbox{BH}\subsetneq \mathcal{F}_1 \subsetneq \mathcal{F}_2 \subsetneq \mbox{Eigen-CG}.\] 
However, perhaps surprisingly, we show that the closures of $ \mbox{BH}$, $\mathcal{F}_1$, and $\mathcal{F}_2$ are equal.
This provides a new account for how general the BH inequalities are. We conjecture that the closure of $\mbox{Eigen-CG}$ is also equal to that of $ \mbox{BH}$, but we have not yet been able to formally prove this.

In terms of the limits of BH inequalities, it is known that there are facets of $\mbox{BQP}_n$ for $n=6$ that cannot be described as \eqref{eq:BHineq}~\cite{Grishukhin90_all-facets-of-cut7}.
Here, we show that $\mbox{Eigen-CG}$ also fails to produce all the facets of $\mbox{BQP}_n$ for $n=6$.
We do so by exhibiting multiple facets of $\mbox{BQP}_6$ with formal proofs that they cannot be recovered as Eigen-CG inequalities.
In contrast, for $n\leq 5$, Grishukhin~\cite{Grishukhin90_all-facets-of-cut7} proved
that all facets of $\mbox{BQP}_n$ are BH inequalities.

As a last theoretical contribution, we analyze the depth of \mbox{Eigen-CG} cuts: if a cut $\alpha^\T z \leq \beta$ separates a point $z^*$, its depth is defined as $(\alpha^\T z^* - \beta)/\|\alpha\|_2$ \cite{dey2018theoretical}.
In this regard, we show that the depth of Eigen-CG cuts degrades with the density of $v$, whenever the vector $(\tilde x,\tilde X)$ to cut satisfies the SDP constraint \eqref{SDP-constraint} and the McCormick inequalities \eqref{eq:McCormick}.
This provides a complementary point of view to what is observed in practice: that sparse inequalities, such as triangle inequalities, can significantly enhance an SDP relaxation of \eqref{QCQP-Ext}.

Finally, based on our theoretical findings, we devise a computational approach to generate sparse BH inequalities. We confirm, empirically, that the effectiveness of these cuts rapidly degrades with their density. We also demonstrate that cuts that are slightly denser than triangle inequalities can significantly improve the quality of the dual bounds.

\section{Related literature} \label{sec:literature}

There is a large body of literature on valid inequalities for QCQPs.
Those that are valid for BQP have received significant attention~\cite{letchford2022qubo},
including from studies of the \emph{cut polytope}~\cite{Barahonaj83_max-cut-problem_cut-polytope},
which is equivalent to BQP (via a linear transformation known as the \emph{covariance map}~\cite{DeSimone89_cut-polytope-covariance-map}).
Many of these belong to the family of BH inequalities~\cite{boros1993cut},
such as McCormick, 
triangle, clique, and hypermetric correlation~\cite{TylkinDeza60,Kelly60_hypermetric-spaces} inequalities.
BH inequalities map to (so-called) \emph{rounded psd} or \emph{$k$-gonal inequalities} for the cut polytope~\cite{deza1997geometry,AviUme03_stronger-LP-relaxations-max-cut}.

Though there are infinitely many BH inequalities, the closure is a polytope~\cite{letchford2012binary}.
Boros and Hammer~\cite{boros1993cut} developed polynomial-time solvable separation of certain subclasses of BH inequalities by solving a minimum weight spanning tree problem.
In these lines, \cite{yajima1998polyhedral} and \cite{letchford2014new} have also considered efficient separation routines over structured classes of BH-type inequalities. In \cite{BonGunLi18}, the authors study the computational impact of BQP-based cutting planes (some of which are special cases of BH inequalities), and in \cite{SheFra02}, the authors considered a special class of SDP-based cutting planes that produce inequalities related to BH.
The work of \cite{macambira2000edge,sorensen2004new} also developed separation routines for BQP-related inequalities in the context of weighted clique problems.
Other specialized approaches for generating BH inequalities can be found in the literature through the analogous procedures for the cut polytope; e.g. \cite{deza1997geometry}.
BH inequalities are generalized further by \emph{gap inequalities}~\cite{LauPol96_gap-inequalities-for-the-cut-polytope,GalKapLet11_gap-inequalities-for-nonconvex-miqp}.

CG rounding has also been considered in this context.
The previously mentioned work by \cite{letchford2014new}, for example, considers a $\{0,1/2\}$-CG procedure \cite{CapFis96_zero-half-cuts} to generate valid BH inequalities.
It is known that the set of CG cuts valid for BQP is dominated by odd cycle inequalities
(complete graph case given by Boros, Crama, and Hammer~\cite{boros1992chvatal}, sparse case proved by Bonami, G\"{u}nl\"{u}k, and Linderoth~\cite{BonGunLi18}).

Also related to our setting, the CG rounding procedure has been considered for integer conic programs~\cite{CezIye05_cuts-for-mixed-0-1-conic-programming} and SDPs~\cite{IyeCez01_cutting-planes-for-mixed-0-1-sdps,deMSot24_integrality-in-sdps,de2025chvatal}.
Eigen-CG inequalities can be viewed within the framework by de Meijer and Sotirov~\cite{de2025chvatal}, though we identify inequalities based solely on negative eigenvalues rather than via aggregations of the constraints.
The CG procedure has also been used for deriving cuts for polynomial optimization~\cite{DelDiG21_Chvatal-rank-in-binary-polynomial-optimization} and general convex compact sets~\cite{dadush2014chvatal}.

The facets of BQP$_6$ have been enumerated and classified~\cite{DezaLaurent88_facets-of-complete-cut-cone,DezLau92_facets-for-the-cut-cone,Grishukhin90_all-facets-of-cut7},
and it is known that $n=6$ is the smallest value for which BQP$_n$ has facets that are not hypermetric correlation inequalities.
This implies that the family of BH inequalities yields all facets of BQP$_n$ for $n\leq 5$.
For larger dimensions, the set of facets is known for the cut polytope from the complete graph on 8 vertices, and for certain other symmetric graphs~\cite{DezSik16_enumeration-of-facets-of-cut-polytopes}.
For more extensive coverage of related work on valid inequalities for BQP,  we refer the interested reader to the recent survey by Letchford~\cite{letchford2022qubo}.

\section{Eigen-CG inequalities and nested families}\label{sec:eigencg}

In \eqref{eq:eigenCGcut} and \eqref{eq:BHineq}, we have defined the Eigen-CG and BH inequalities. Let us first argue why the Eigen-CG cuts are valid.

\begin{lemma}
For every $(v_0 , v)\in \mathbb{R}^{n+1}$, inequality \eqref{eq:eigenCGcut} is valid for $\mbox{QPB}_n$.
\end{lemma}
\begin{proof}
We know the eigen-cuts \eqref{eq:eigencut} are valid for $\mbox{QPB}_n$.
If we further assume $(x,X)$ is binary, we obtain that
\[\sum_{1\leq i < j \leq n} 2v_iv_j X_{i,j} + \sum_{i=1}^n (v_i^2 + 2v_i v_0) x_i + v_0^2 \geq 0\]
is valid for $\mbox{BQP}_n$. Furthermore, we can round coefficients safely (first round up the non-constant terms, and then round down the constant term) and obtain that \eqref{eq:eigenCGcut} is valid for $\mbox{BQP}_n$.
The result then follows from Proposition \ref{prop:burerletchford}.\qed
\end{proof}

In order to compare Eigen-CG and BH inequalities, and define other 
families of inequalities, we use the following definition.

\begin{definition}
Given $(v_0 , v)\in \mathbb{R}^{n+1}$, we define $\mbox{E-CG}(v_0,v)$ as the coefficients of \eqref{eq:eigenCGcut} obtained using
$(v_0,v)$. 
This means that $\mbox{E-CG}(v_0,v) = (\alpha, \beta, \gamma) \in \mathbb{R}^{n+{n\choose 2}+1}$ where 
\begin{align*}
\alpha_i &= \lceil v_i^2 + 2v_i v_0 \rceil && \forall i \in [n]\\
\beta_{ij} &= \lceil 2v_iv_j \rceil && \forall i<j \textup{ and } i,j \in [n] \\
\gamma &= \lfloor v_0^2 \rfloor.
\end{align*}
\end{definition}
We study the following restricted families of Eigen-CG inequalities
\begin{align*}
\mathcal{F}_0&=\bigl\{\mbox{E-CG}(v_0,v):\ v\in\mathbb{Z}^n,\ v_0 + 1/2\in\mathbb{Z}\bigr\},\nonumber\\[2pt]
\mathcal{F}_1&=\bigl\{\mbox{E-CG}(v_0,v):\ v\in\mathbb{Z}^n,\ 2v_iv_0\in\mathbb{Z}\ \forall i\bigr\},\nonumber\\[2pt]
\mathcal{F}_2&=\bigl\{\mbox{E-CG}(v_0,v):\ v_i^2\in\mathbb{Z} \ \forall i,\ 2v_iv_j\in\mathbb{Z}\ \forall i\ne j,\ 2v_iv_0\in\mathbb{Z}\ \forall i\bigr\}.
\end{align*}
We note that the $\mathcal{F}_i$ families were constructed based on conditions that make the rounding step not needed 
for the non-constant terms in $\mbox{E-CG}(v_0,v)$. They clearly satisfy $\mathcal{F}_0\subseteq \mathcal{F}_1 \subseteq \mathcal{F}_2 \subseteq \mbox{E-CG}$.

\subsection{Direct comparison of restricted families}

In this subsection, we show that these families are indeed different. In the following subsection, we analyze their closures.

\begin{lemma}\label{lemma:strictcontainment}\hspace{0.1cm}
\begin{enumerate}
\item[(i)] There is a $(v_0,v)\in \mathbb{R}^{n+1}$ such that $\mbox{E-CG}(v_0,v) \in \mathcal{F}_1$, and that there is no scalar $\lambda > 0$ such that
$\lambda\cdot\mbox{E-CG}(v_0,v) \in \mathcal{F}_0$.
\item[(ii)] There is a $(v_0,v)\in \mathbb{R}^{n+1}$ such that $\mbox{E-CG}(v_0,v) \in \mathcal{F}_2$, and that there is no scalar $\lambda > 0$ such that
$\lambda\cdot\mbox{E-CG}(v_0,v) \in \mathcal{F}_1$.
\item[(iii)] There is a $(v_0,v)\in \mathbb{R}^{n+1}$ such that there is no scalar $\lambda > 0$ such that
$\lambda\cdot\mbox{E-CG}(v_0,v) \in \mathcal{F}_2$.
\end{enumerate}
\end{lemma}

\begin{proof}
\noindent (i) Let $v=(2,-4)$, $v_0=3/4$. Then, $2v_1v_0=3$ and $2v_2v_0=-6$, so $\mbox{E-CG}(v_0,v)\in\mathcal{F}_1$. 
Since $\lfloor v_0^2\rfloor=0$, the resulting Eigen-CG inequality is
\[
-16X_{12}+7x_1+10x_2\ge 0.
\]
Let us assume there is $\lambda > 0$ such that $\lambda\cdot\mbox{E-CG}(v_0,v) \in \mathcal{F}_0$. This means that for some $(\tilde{v},\tilde{v}_0)$, with $\tilde v \in \mathbb{Z}^n$ and $\tilde{v}_0 +1/2 \in \mathbb{Z}$, we have $-16\lambda = 2\tilde{v}_1\tilde{v}_2$, $7\lambda = \tilde{v}_1^2 + 2\tilde{v}_1\tilde{v}_0$, $10\lambda = \tilde{v}_2^2 + 2\tilde{v}_2\tilde{v}_0$, and $ \lfloor \tilde{v}_0^2 \rfloor = 0$.
The last condition implies that $ \tilde{v}_0 \in (-1,1)$, and thus $ \tilde{v}_0 = \pm 1/2$. Without loss of generality, we can assume $ \tilde{v}_0 =  1/2$. Note that these conditions imply
\[-8(\tilde{v}_1 + 1) = 7\tilde{v}_2, \qquad  -8(\tilde{v}_2 + 1) = 10 \tilde v_1.\] 
It can be easily checked that this linear system has a unique solution, which is not integral.\\

\noindent (ii) Let $v=(5\sqrt{2},-10\sqrt{2})$, $v_0= 7\sqrt{2}/5$. Then $v_1^2= 50$, $v_2^2= 200$,
$2v_1v_2=-200$, $2v_1v_0= 28$, and $2v_2v_0= -56$, so $\mbox{E-CG}(v_0,v)\in\mathcal{F}_2$. 
The Eigen-CG inequality in this case is
\[
-200X_{12}+78x_1+144x_2 + 3\ge 0.
\]
Let us assume there is $\lambda > 0$ such that $\lambda\cdot\mbox{E-CG}(v_0,v) \in \mathcal{F}_1$. This means that for some $(\tilde{v},\tilde{v}_0)$, with $\tilde v \in \mathbb{Z}^n$ and $2 \tilde{v}_i\tilde{v}_0  \in \mathbb{Z}$, we have $-200\lambda = 2\tilde{v}_1\tilde{v}_2$, $78\lambda = \tilde{v}_1^2 + 2\tilde{v}_1\tilde{v}_0$, and $144\lambda = \tilde{v}_2^2 + 2\tilde{v}_2\tilde{v}_0$.

Note that $\tilde{v}_1 \neq 0$ and $ \tilde{v}_2 \neq 0$. This implies 
\[\frac{78\lambda - \tilde{v}_1^2 }{2\tilde{v}_1} = \frac{144\lambda - \tilde{v}_2^2}{2\tilde{v}_2}  \Rightarrow (144 \tilde{v}_1 - 78 \tilde{v}_2) \lambda = \tilde{v}_1 \tilde{v}_2 (\tilde{v}_1 - \tilde{v}_2).\]

Using that $\lambda = -\tilde{v}_1 \tilde{v}_2 /100$ we get $\tilde{v}_2 = 244 \tilde{v}_1 / 178$. But note that this implies $\tilde{v}_1 \tilde{v}_2 \geq 0$, which contradicts $\lambda > 0$.\\

\noindent (iii) Let $v=(1,-\sqrt{2},\sqrt{3})$, $v_0=0$. The Eigen-CG inequality in this case is

\[-2X_{12} + 4X_{13} - 4X_{23} + x_1 + 2x_2 + 3x_3 \geq 0.\]
As before, we assume there is $\lambda > 0$ such that $\lambda\cdot\mbox{E-CG}(v_0,v) \in \mathcal{F}_2$. Thus, we have
$-2\lambda = 2\tilde{v}_1\tilde{v}_2$, $4\lambda = 2\tilde{v}_1\tilde{v}_3$, $-4\lambda = 2\tilde{v}_2\tilde{v}_3$, $\lambda = \tilde{v}_1^2 + 2\tilde{v}_1\tilde{v}_0$, $2\lambda = \tilde{v}_2^2 + 2\tilde{v}_2\tilde{v}_0$, and $3\lambda = \tilde{v}_3^2 + 2\tilde{v}_3\tilde{v}_0$.
Since $\lambda > 0$, no $\tilde{v}_i$, $i\geq 1$, can be 0. Replacing $\lambda$ throughout, we obtain the following linear equations
\begin{align*}
-\tilde{v}_2 = \tilde{v}_1 + 2\tilde{v}_0\quad&\quad \frac{1}{2}\tilde{v}_3 = \tilde{v}_1 + 2\tilde{v}_0\\
-\frac{1}{2}\tilde{v}_3 = \frac{1}{2}(\tilde{v}_2 + 2\tilde{v}_0)\quad&\quad  -\tilde{v}_1= \frac{1}{2}(\tilde{v}_2 + 2\tilde{v}_0)\\
-\frac{1}{2}\tilde{v}_2 = \frac{1}{3}(\tilde{v}_3 + 2\tilde{v}_0)\quad&\quad \frac{1}{2}\tilde{v}_1 = \frac{1}{3}(\tilde{v}_3 + 2\tilde{v}_0).
\end{align*}
It can be verified that the only solution for this system is $\tilde{v}_i = 0$ for all $i$, a contradiction.\qed
\end{proof}

\subsection{Closure comparison of restricted families}\label{sec:BHequivalence}

We begin by noting that the BH inequalities \eqref{eq:BHineq} correspond to the $\mathcal{F}_0$ family.
We remark that this is known from the cut-polytope literature \cite{letchford2012binary}, but for completeness, we state a self-contained proof here.

\begin{lemma}\label{lem:BHinCG}
The $\mathcal{F}_0$ family is exactly the family of Boros--Hammer inequalities.
\end{lemma}
\begin{proof}
Fix $(w_0,w)\in \mathbb{Z}^{n+1}$. We claim that $\mbox{E-CG}(w_0 - 1/2, w)\in \mathcal{F}_0$ yields the BH inequality defined by $(w_0,w)$.
The terms associated with $X_{ij}$ are directly obtained.
For each $i$, $w_i^2+2w_i(w_0 -1/2)=w_i^2+2w_iw_0-w_i$, which matches the coefficient of $x_i$ in \eqref{eq:BHineq}. 
Finally, 
\[ \left(w_0 - \frac{1}{2}\right )^2=w_0(w_0-1)+\frac{1}{4}  \Rightarrow \left\lfloor \left(w_0 - \frac{1}{2}\right )^2\right\rfloor=w_0(w_0-1). \]
The proof to show that any $\mbox{E-CG}(v_0, v)\in \mathcal{F}_0$ defines a BH inequality is the same.\qed
\end{proof}

In principle, and based on Lemma \ref{lemma:strictcontainment}, one could think that the additional freedom of $\mathcal{F}_1$ or
$\mathcal{F}_2$ could yield more expressiveness than BH in terms of their closures.
In what follows, we show that this is not the case by showing that \emph{every} inequality in $\mathcal{F}_2$ is implied by a conic 
combination of BH inequalities.\\

Before showing the result, we need the following technical condition.
\begin{proposition}\label{prop:irrational}
Let $v\in \mathbb{R}^n \setminus \{0\}$ and $v_0\in \mathbb{R}$, and suppose $v_i^2\in\mathbb{Z}$ for all $i$,
$2v_iv_j\in\mathbb{Z}$ for all $i\neq j$, and $2v_iv_0\in\mathbb{Z}$ for all $i$.
Then, there exists $p\in\mathbb{R}$ with $p^2\in\mathbb{Z}$, $2v_0p\in\mathbb{Z}$,
and $v_i/p\in\mathbb{Z}$ for all $i$.
\end{proposition}
\begin{proof}
Let $a_i:=v_i^2$. 
Using the prime factorization of $a_i$, we can write $a_i=s_i^2 d_i$ with $d_i\in \mathbb{Z}$ square-free\footnote{A square-free integer is an integer whose prime factorization has no exponent greater than or equal to 2. By convention, 1 is square-free.} and $s_i\in\mathbb{Z}$.
Furthermore, we assume the sign of $s_i$ is such that $v_i = s_i \sqrt{d_i}$.
Then, $2v_iv_j=2s_is_j\sqrt{d_id_j}\in\mathbb{Z}$ forces $d_id_j$ to be a square. Since the $d_i$ are square-free, it must hold that $d_i=d_j =: d$ for every $i,j$. Thus, $v_i = s_i \sqrt{d}$, with not all $s_i$ zero.

Let $g:=\gcd\{(s_i)_{i=1}^n\}>0$ and define $p:=g\sqrt{d}$.
Then, $p^2=g^2 d\in\mathbb{Z}$ and $v_i/p= s_i/g \in\mathbb{Z}$. 
We are only missing to show $2v_0p\in\mathbb{Z}$.
From $2v_iv_0\in\mathbb{Z}$ we have $2v_0s_i\sqrt{d}\in\mathbb{Z}$ for all $i$.
Let $q_i = s_i/g\in \mathbb{Z}$. By definition of $g$, it must hold that $\gcd\{(q_i)_{i=1}^n\}=1$. Then,
\[2v_0s_i\sqrt{d} = 2v_0q_i g\sqrt{d} = 2v_0q_i p \in\mathbb{Z}\qquad \forall i.\]
Let us write $2v_0 p = a/b$, with $a,b\in \mathbb{Z}$ coprimes and $b\geq 1$. Since $q_i\in \mathbb{Z}$, it must hold that $b$ divides all $g_i$. But since $\gcd\{(q_i)_{i=1}^n\}=1$, we conclude that $b=1$. This means $2v_0 p = a \in \mathbb{Z}$.\qed
\end{proof}

\begin{lemma}\label{thm:CGinBH}
Take $(v_0, v)$ such that $\mbox{E-CG}(v_0,v)\in \mathcal{F}_2$. Then,
$\mbox{E-CG}(v_0,v)$ is implied by a nonnegative combination of at most two BH inequalities.
\end{lemma}
\begin{proof}
By definition of $\mathcal{F}_2$, the vector $(v_0,v)$ satisfies the conditions of Proposition \ref{prop:irrational}.
Thus, we can take $p\in\mathbb{R}$ with $p^2\in\mathbb{Z}$, $2v_0p\in\mathbb{Z}$,
and $r_i := v_i/p\in\mathbb{Z}$ for all $i$. \\

Now, take $a\in \mathbb{Z}$ such that
\[a - \frac{1}{2} \leq \frac{v_0}{p} \leq a + \frac{1}{2}.\]
We consider the following two inequalities: the inequality defined by $\mbox{E-CG}( a-1/2,r)$, i.e.,
\begin{equation} \label{minusineq}
\sum_{i < j} 2 r_i r_j X_{ij} + \sum_i \left( r_i^2 + 2 r_i \left( a - \frac{1}{2} \right) \right) x_i + \left\lfloor \left( a - \frac{1}{2} \right)^2 \right\rfloor \geq 0
\end{equation}
and the inequality defined by $\mbox{E-CG}( a+1/2,r)$, i.e.,
\begin{equation} \label{plusineq}
\sum_{i < j} 2 r_i r_j X_{ij} + \sum_i \left( r_i^2 + 2 r_i \left( a + \frac{1}{2} \right) \right) x_i + \left\lfloor \left( a + \frac{1}{2} \right)^2 \right\rfloor \geq 0.
\end{equation}

Note that both of these Eigen-CG inequalities are actually BH inequalities, by virtue of Lemma \ref{lem:BHinCG}.
We claim that these two inequalities imply the inequality defined by $\mbox{E-CG}(v_0,v)$.

Define the multipliers
\begin{align*}
\alpha^- & = \frac{1}{2}p^2 + ap^2 - v_0 p  && \text{ for \eqref{minusineq}}\\
\alpha^+ & = \frac{1}{2}p^2 - ap^2 + v_0 p.&& \text{ for \eqref{plusineq}}
\end{align*}
By construction of $a$, we have $\alpha^-, \alpha+ \geq 0$. Moreover, $\alpha^+ + \alpha^- = p^2$, and thus the combination of the coefficients associated with $X_{ij}$ yields precisely
\[2(\alpha^+ + \alpha^-)r_ir_j = 2p^2 r_ir_j = 2v_iv_j.\]

As for the coefficients of $x_i$, the part associated with $r_i^2$ follows the same reasoning. For the rest, we analyze
\begin{align*}
\alpha^- \left( a - \frac{1}{2} \right) + \alpha^+ \left( a + \frac{1}{2} \right)  &= ap^2  + \frac{1}{2} \left( \alpha^+ - \alpha^- \right) \\
&= ap^2 + \frac{1}{2} \left( -2ap^2 + 2v_0p\right) \\
&= v_0 p,
\end{align*}
which is, again, what we want since $2r_iv_0 p = 2v_iv_0$. We are just missing the constant part.
Note that, since $a\in \mathbb{Z}$,
\begin{align*}
\left \lfloor \left(a - \frac{1}{2}\right)^2\right\rfloor &= \left \lfloor a^2 - a + \frac{1}{4}\right\rfloor = a^2 - a \\
\left \lfloor \left(a + \frac{1}{2}\right)^2\right\rfloor &= \left \lfloor a^2 + a + \frac{1}{4}\right\rfloor = a^2 + a,
\end{align*}
thus,
\begin{align*}
\alpha^- \left \lfloor \left(a - \frac{1}{2}\right)^2\right\rfloor + \alpha^+  \left \lfloor \left(a + \frac{1}{2}\right)^2\right\rfloor &= \alpha^- (a^2 - a) + \alpha^+( a^2 + a)\\
&= a^2p^2 + a (-2ap^2 + 2v_0p).
\end{align*}
The last expression is an integer since $a\in \mathbb{Z}$, $p^2\in \mathbb{Z}$, and $2v_0p \in \mathbb{Z}$. 
Therefore, if we can show
\begin{equation}\label{lastineq}
a^2p^2 + a (-2ap^2 + 2v_0p) \leq  v_0^2 ,
\end{equation}
we immediately have 
\[a^2p^2 + a (-2ap^2 + 2v_0p) \leq  \lfloor v_0^2 \rfloor, \]
and this would prove that the implied inequality is at least as strong as $\mbox{E-CG}(v_0,v)$. Let us show \eqref{lastineq} then. Rearranging terms, this becomes
\[0 \leq v_0^2 - 2v_0 a p + a^2p^2  = (v_0 - ap)^2,\]
which shows what we want.\qed
\end{proof}

We finalize this subsection with a conjecture.

\begin{conj}\label{conjecture}
For any $(v_0, v)\in \mathbb{R}^{n+1}$, the inequality
defined by  $\mbox{E-CG}(v_0,v)$ is implied by a nonnegative combination of BH inequalities.
\end{conj}

To date, every computational test we have devised suggests that this conjecture is true, but we have not been able to find a proof.

\subsection{Limitations of Eigen-CG cuts: missing facets of $\mbox{BQP}_6$}\label{sec:dim6}

Via exhaustive enumeration, we know that BH inequalities capture all facets of $\mbox{BQP}_5$; in particular, Eigen-CG inequalities describe $\mbox{BQP}_5$.
This is no longer true for $n=6$ \cite{Grishukhin90_all-facets-of-cut7,deza1997geometry}.
In this subsection, we analyze the facets of $\mbox{BQP}_6$ that are \emph{not} representable by any Eigen-CG inequality.

In total, $\text{BQP}_6$ has 116,764 facets \cite{deza1997geometry}.
Among these, our computational analysis shows that 3,676 inequalities can be represented as Eigen-CG inequalities (furthermore, BH), while the remaining 113,088 facets are \emph{not} of the Eigen-CG type.
Therefore, while Eigen-CG inequalities form a nontrivial and expressive family, they account for only a small fraction of the overall facet structure.
As mentioned in the introduction, the fact that BH is not sufficient for $\mbox{BQP}_6$ was known. But since we do not know if Eigen-CG has the same expressiveness as BH (Conjecture \ref{conjecture}), one could conceive Eigen-CG to describe facets of $\mbox{BQP}_6$ that BH cannot.

How do we \emph{know} that certain inequalities are not Eigen-CG? In what follows, we show examples of facets with different proof strategies that can formally certify that they are indeed not representable as Eigen-CG inequalities.
These different proof strategies can actually be
used to show, without relying on numerical methods, that the aforementioned 113,088 facets are not Eigen-CG.

Let us begin with a facet that can be handled with a simple argument.
\begin{proposition}\label{prop:dim6_sign}
The following facet\footnote{This facet was obtained via exhaustive enumeration.} of $\mbox{BQP}_6$
\begin{align}\label{eq:dim6facet}
&2 - 2x_1 - x_2 + x_3 - 2x_4 + 3x_5 - x_6 + 2X_{12} - X_{13} - X_{23} + 2X_{14} + X_{24}\nonumber\\
&\qquad - 2X_{15} - X_{25} - X_{45} + X_{16} + X_{36} + X_{46} - X_{56} \ge 0
\end{align}
cannot be represented as an Eigen-CG inequality.
\end{proposition}
\begin{proof}
Let us suppose there exists $(v_0,v)\in \mathbb{R}^{n+1}$ and $\lambda > 0$ such that $\lambda\cdot \mbox{E-CG}(v_0,v)$ corresponds to \eqref{eq:dim6facet}.
The zero coefficient on $X_{26}$ in \eqref{eq:dim6facet} indicates $v_2v_6\leq0$. The terms $2X_{12}$ and $X_{16}$ imply that $v_1v_2 > 0$ and $v_1 v_6 > 0$. This implies that $\mbox{sign}(v_2) = \mbox{sign}(v_1) = \mbox{sign}(v_6)$.
Thus, to satisfy $v_2v_6\leq0$, either $v_2 = 0$ or $v_6 = 0$. This leads to a contradiction.\qed
\end{proof}

Note that since the inequality of Proposition~\ref{prop:dim6_sign} is a facet, it cannot be \emph{implied} by other valid inequalities.\\

We illustrate a different proof strategy for non-Eigen-CG inequalities in the following proposition.

\begin{proposition}
\label{example_dim6facet_lambda}
The following facet\footnote{This facet was obtained via exhaustive enumeration.} of $\mbox{BQP}_6$
\begin{align}\label{eq:dim6facet_lambda}
5 &- 4x_1 - 4x_2 + 4x_3 - 5x_4 + 4x_5 - 3x_6 \nonumber\\
& + 3X_{12} - 2X_{13} - 2X_{23} + 5X_{14} + 5X_{24} - 3X_{34} -  2X_{15} - 2X_{25} \\
&  + 1X_{35} - 3X_{45} + 2X_{16} + 2X_{26} - 1X_{36} + 3X_{46} - 1X_{56} \ge 0\nonumber
\end{align}
cannot be represented as an Eigen-CG inequality.
\end{proposition}

Before presenting the proof, we remark that the facet \eqref{eq:dim6facet_lambda} passes the sign pattern ``test'' of Proposition \ref{prop:dim6_sign}. This indicates that not all non-Eigen-CG facets can be certified simply using sign patterns. Thus, we require a different proof.

\begin{proof}
Suppose, by contradiction, that \eqref{eq:dim6facet_lambda} can be represented as a (possibly rescaled) Eigen-CG inequality.
Then there exist $(v_0,v)\in\mathbb{R}^{n+1}$ and an integer $\lambda>0$ such that $(1/\lambda)\cdot \mbox{E-CG}(v_0,v)$ yields the coefficients of \eqref{eq:dim6facet_lambda}.

In particular, the coefficients $3X_{12}$, $X_{35}$, $-2X_{13}$, and $-2X_{25}$ imply that
\begin{align*}
\ceil{2v_1 v_2}=3\lambda,\, \ceil{2v_3 v_5}=\lambda,\, \ceil{2v_1 v_3}=-2\lambda,\, \ceil{2v_2 v_5}=-2\lambda.
\end{align*}
 Let $\bar{\beta}_{ij} := 2v_i v_j$. We should have 
\begin{align*}
  3\lambda-1<&\,\bar{\beta}_{12}\leq 3\lambda,\\
  \lambda-1<&\,\bar{\beta}_{35}\leq \lambda,\\ 
  -2\lambda-1<&\,\bar{\beta}_{13}\leq -2\lambda, \\
  -2\lambda-1<&\,\bar{\beta}_{25}\leq -2\lambda, 
\end{align*}
which implies that $\bar\beta_{12}\bar\beta_{35} \leq 3\lambda^2$, $\bar\beta_{13}\bar\beta_{25} \geq 4\lambda^2$.
This contradicts that
$\bar\beta_{12}\bar\beta_{35}=\bar\beta_{13}\bar\beta_{25}$. Hence, the facet \eqref{eq:dim6facet_lambda} cannot be represented as an Eigen-CG inequality.\qed
\end{proof}

As opposed to the sign pattern argument of Proposition~\ref{prop:dim6_sign}, Proposition~\ref{example_dim6facet_lambda} exploits incompatibilities arising from $2\times 2$ multiplicative relationships among bilinear coefficients. 
These two arguments combined certify many, but not all of the non-Eigen-CG facets of $\mathrm{BQP}_6$.
In particular, if we automate these two arguments, then among the 113,088 facets that are not of the Eigen-CG type, all but 37,338 can be certified to be non-Eigen-CG using either the sign-pattern argument or the $2\times 2$ argument.

Certifying the remaining non-Eigen-CG facets of $\mathrm{BQP}_6$ requires analyzing a larger and more intricate system of relations among the coefficients.
We did so for each of the remaining cases (after reducing cases via symmetries among the facets).
We omit each case in detail for the sake of brevity, but we illustrate one of these remaining cases to give the reader an idea of how a proof can be obtained.

\begin{proposition}\label{prop:dim6_general}
The following facet\footnote{This facet was obtained via exhaustive enumeration.} of $\mbox{BQP}_6$
\begin{align}\label{eq:dim6facet_general}
&1- x_1+ 2x_3+ x_5 +2X_{12}-X_{13}-3X_{23}+X_{14}+2X_{24}-2X_{34}\nonumber\\
&\quad -X_{15}-2X_{25}+2X_{35}-X_{45} +X_{16}+2X_{26}-2X_{36}+X_{46}-X_{56} \ge 0
\end{align}
cannot be represented as an Eigen-CG inequality.
\end{proposition}
\begin{proof}
Suppose, by contradiction, that \eqref{eq:dim6facet_general} can be represented as a (possibly rescaled) Eigen-CG inequality.
Then there exist $(v_0,v)\in\mathbb{R}^{n+1}$ and an integer $\lambda>0$ such that $(1/\lambda)\cdot \mbox{E-CG}(v_0,v)$ yields the coefficients of \eqref{eq:dim6facet_general}.

We focus on the following six coefficients:
\[
c_{12} = 2,\,
c_{34} = -2,\,
c_{56} = -1,\,
c_{13} = -1,\,
c_{25} = -2,\,
c_{46} = 1.
\]
Let $\bar{\beta}_{ij} := 2 v_i v_j$. Then, similarly to the proof of Proposition \ref{example_dim6facet_lambda}, $\ceil{\bar{\beta}_{ij}} = \lambda c_{ij}$, which implies
\[
\bar{\beta}_{ij} \in \big(\lambda c_{ij} - 1, \lambda c_{ij}\big]
\qquad\text{for each pair } (i,j).
\]
In addition, the following must hold
\begin{align}
\bar{\beta}_{12}\,\bar{\beta}_{34}\,\bar{\beta}_{56}=
\bar{\beta}_{13}\,\bar{\beta}_{25}\,\bar{\beta}_{46}=8 v_1 v_2 v_3 v_4 v_5 v_6.
\label{eq:beta-matching-identity}
\end{align}
and since there are coefficients $c_{ij}$ such that $|c_{ij}| = 1$ (e.g. $\{ij\} = \{46\}$), we must have $\lambda \in \mathbb{Z}$; in particular, $\lambda \geq 1$.

Let us define $P_1 := \bar{\beta}_{12}\,\bar{\beta}_{34}\,\bar{\beta}_{56}$ and $P_2 := \bar{\beta}_{13}\,\bar{\beta}_{25}\,\bar{\beta}_{46}$.
We show that for every integer $\lambda \ge 1$, the feasible ranges of $P_1$ and $P_2$ implied by the ceiling relations are disjoint, and therefore \eqref{eq:beta-matching-identity} cannot hold. We split the proof into two parts. 

\noindent\textbf{Part 1.} From $c_{12}=2$, $c_{34}=-2$, and $c_{56}=-1$, we obtain
\[
\begin{aligned}
\bar{\beta}_{12} \in (2\lambda - 1,\; 2\lambda],\,
\bar{\beta}_{34} \in (-2\lambda - 1,\; -2\lambda],\,
\bar{\beta}_{56} \in (-\lambda - 1,\; -\lambda].
\end{aligned}
\]
Thus, since $\lambda \geq 1$, we have
\begin{align*}
P_1 \in \big(P_1^{\min},\,P_1^{\max}\big]
:=
\big((2\lambda-1)\,2\lambda^2,\;2\lambda(2\lambda+1)(\lambda+1)\big].
\end{align*}
Similarly, from $c_{13}=-1$, $c_{25}=-2$, and $c_{46}=1$, we have
\[
\begin{aligned}
\bar{\beta}_{13} \in (-\lambda - 1,\; -\lambda],\,
\bar{\beta}_{25} \in (-2\lambda - 1,\; -2\lambda],\,
\bar{\beta}_{46} \in (\lambda - 1,\; \lambda].
\end{aligned}
\]
from where we obtain
\begin{align*}
P_2 \in \big[P_2^{\min},\,P_2^{\max}\big]
:=
\big[2\lambda^3 - 2\lambda^2,\;\lambda(\lambda+1)(2\lambda+1)\big].
\end{align*}
For \eqref{eq:beta-matching-identity} to hold, it is necessary that $\big(P_1^{\min},P_1^{\max}\big] \cap \big[P_2^{\min},P_2^{\max}\big] \neq \varnothing$.
A sufficient condition for impossibility is $P_1^{\min} > P_2^{\max}$. Indeed, we have
\begin{align*}
P_1^{\min} - P_2^{\max}
= (2\lambda-1)\,2\lambda^2 - \lambda(\lambda+1)(2\lambda+1)= \lambda\big(2\lambda^2 - 5\lambda - 1\big), 
\end{align*}
which is strictly positive for all integers $\lambda \ge 3$.
Therefore, the facet \eqref{eq:dim6facet_general} cannot be represented as a multiple of an Eigen-CG inequality for $\lambda \ge 3$.

\noindent\textbf{Part 2.} 
It remains to consider the cases $\lambda\in\{1,2\}$.
Let us first consider $\lambda=1$.
From $\lfloor v_0^2\rfloor=1$, we have $1 \le v_0^2 < 2$. Without loss of generality, we assume $v_0>0$.
Since the term $x_5$ has coefficient 1, we have $v_5^2+2v_5v_0 \in (0,1]$. 
The lower bound of 0 yields two cases:
\[
\text{(i)}\quad 0< v_5, \qquad
\text{(ii)}\quad v_5 < -2v_0.\]
Consider case (ii). From the coefficient $-2X_{25}$, we have $v_2 v_5 <0$. Since $v_5<0$ in this case, it follows that $v_2>0$. However, from the coefficient $0x_2$, we have  $v_2^2+2v_2v_0 \in (-1,0]$. However, $v_2^2+2v_2v_0$ is strictly positive since when $v_2>0, v_0>0$. A contradiction.
Hence, case (ii) is impossible.

Consider now case (i): $0 < v_5$. From $-2X_{25}$, we have $\bar{\beta}_{25}=2v_5v_2 \in (-3,-2]$.
This implies $v_2<0$ and thus $2v_5|v_2| \ge 2$. Additionally, recall that $v_5^2 + 2v_5v_0 \leq 1$, which implies $v_5 \le -v_0+\sqrt{v_0^2+1}$ and thus
\begin{align*}
|v_2|\geq \frac{1}{v_5} \geq 
\frac{1}{-v_0+\sqrt{v_0^2+1}}
=v_0+\sqrt{v_0^2+1}> 2v_0.
\end{align*}

From the coefficient $0x_2$, we have $v_2^2+2v_2v_0=v_2(v_2+2v_0) \in (-1,0]$. But since $v_2<0$, we must have $v_2+2v_0\ge 0$, i.e., $|v_2|=-v_2 \le 2v_0$, which contradicts the previously derived condition $|v_2|> 2v_0$.

A similar argument applies for $\lambda=2$.

This shows that the facet \eqref{eq:dim6facet_general} cannot be represented as an Eigen-CG inequality. \qed
\end{proof}

\subsection{Limitations of Eigen-CG cuts: dense cuts are weak over SDP}\label{sec:depth}

In this section, we show formally that dense Eigen-CG cuts are weak in the presence of an SDP constraint.
This means that, if we have a solution $(\hat{x}, \hat{X})$ that satisfies the McCormick and SDP constraints, any dense Eigen-CG will have a limited depth.

\begin{proposition}\label{proposition_depth}
Consider $(v_0,v)\in \mathbb{R}^{n+1}$. Let $(\hat{x}, \hat{X}) \not\in \mbox{QPB}_n$ be such that it satisfies the SDP constraint \eqref{SDP-constraint} and the McCormick inequalities \eqref{eq:McCormick}. Then, the maximum violation (depth of cut) of $(\hat{x}, \hat{X})$ for the inequality $\mbox{E-CG}(v_0,v)$ is
$$\frac{2}{\sqrt{\|v\|_0 (\|v\|_0 - 2)}},$$
where $\|v\|_0$ is the number of non-zero entries of $v$.
\end{proposition}

\begin{proof}
We know that $\mbox{E-CG}(v_0,v)$ yields the inequality
\begin{equation}\label{eq:CG}
\sum_{i > j} \lceil 2v_iv_j \rceil X_{ij} + \sum_{i} \lceil v_i^2 + 2v_iv_0 \rceil x_i + \lfloor v^2_0 \rfloor \geq 0. 
\end{equation}

Since $(\hat{x}, \hat{X})$ satisfies the SDP constraint \eqref{SDP-constraint}, we have that
\[
\sum_{i > j}  2v_iv_j  \hat{X}_{ij} + \sum_{i} ( v_i^2 + 2v_iv_0 ) \hat{x}_i +  v^2_0  \geq 0. \]

Let $g$ be the vector corresponding to the non-constant part of \eqref{eq:CG}. The depth of $\mbox{E-CG}(v_0,v)$ can be upper bounded as follows:
\begin{align*}
    & \frac{1}{\|g\|}\left( - \lfloor v^2_0 \rfloor - \sum_{i > j} \lceil 2v_iv_j \rceil \hat{X}_{ij} - \sum_{i} \lceil v_i^2 + 2v_iv_0 \rceil \hat{x}_i \right) \\
    \leq \, & \frac{1}{\|g\|}\left( - \lfloor v^2_0 \rfloor - \sum_{i > j}  2v_iv_j  \hat{X}_{ij} - \sum_{i} ( v_i^2 + 2v_iv_0 ) \hat{x}_i \right)\\
    \leq \, & \frac{1}{\|g\|}\left( - \lfloor v^2_0 \rfloor + v^2_0 \right)\\
    \leq\, & \frac{1}{\|g\|_2}.
\end{align*}
Then, observe that

\begin{align*}
\|g\|^2_2 &= \sum_{i < j} \lceil 2v_iv_j \rceil ^2 + \sum_{i} \lceil v_i^2 + 2v_iv_0\rceil^2 \geq  \sum_{i < j} \lceil 2v_iv_j \rceil^2,
\end{align*}
where the inequality follows because we dropped a square term. Now, define
\[P=\{i\,:\, v_i > 0\},\quad N = \{i\,:\, v_i < 0\}. \]
Note that if $i,j\in P$ or $i,j\in N$, then $\lceil 2v_iv_j \rceil^2 \geq 1$. Thus,
\begin{align*}
\|g\|^2_2 &\geq  \sum_{i < j} \lceil 2v_iv_j \rceil^2 \\
&=  \sum_{i < j,\, \lceil 2v_iv_j \rceil^2 \geq 1} \lceil 2v_iv_j \rceil^2 \\
&\geq  \sum_{i < j,\, i,j \in P} \lceil 2v_iv_j \rceil^2 +   \sum_{i < j,\, i,j \in N} \lceil 2v_iv_j \rceil^2 \\
&\geq  {|P|\choose 2} + {|N|\choose 2}.
\end{align*}

To lower bound this last expression, let $k := \|v\|_0$ and  note that $|P|+|N|=k$. Thus, we can consider the following optimization problem:
\begin{align*}
\min\quad & {p \choose 2} + {n \choose 2}\\
\text{s.t.}\quad & p+n = k \\
& p,n \geq 0,
\end{align*}
or equivalently
\begin{align*}
\min\quad & {p \choose 2} + {k-p \choose 2}\\
\text{s.t.}\quad & k-p \geq 0 \\
& p \geq 0.
\end{align*}
This is a unidimensional convex optimization problem, whose optimal solution can be easily checked to be $p^*=k/2$.
Therefore,
\begin{align*}
\|g\|^2_2 &\geq  {p^* \choose 2} + {k-p^* \choose 2} =  \frac{k^2 -2 k}{4}.
\end{align*}
From here, we obtain the result.\qed
\end{proof}

In the case of inequalities of $\mathcal{F}_2$ (which include BH), we can derive an upper bound on their depth that additionally depends on $\|v\|_2$.
\begin{proposition}
\label{proposition_depth-BH}
Consider $(v_0,v)\in \mathbb{R}^{n+1}$ such that $\mbox{E-CG}(v_0,v) \in \mathcal{F}_2$. Let $(\hat{x}, \hat{X}) \not\in \mbox{QPB}_n$ be such that it satisfies the SDP constraint \eqref{SDP-constraint} and the McCormick inequalities \eqref{eq:McCormick}. Then, the maximum violation (depth of cut) of $(\hat{x}, \hat{X})$ for the inequality $\mbox{E-CG}(v_0,v)$ is
$$\frac{1}{\sqrt{(\|v\|_0 - 1)}\cdot\|v\|_2},$$
where $\|v\|_0$ is the number of non-zero entries of $v$.
\end{proposition}
\begin{proof}
We follow the same proof as before, with the following difference:
\begin{eqnarray*}
\|g\|^2_2 &=& \sum_{i < j} (2v_iv_j)^2 + \sum_{i} (v_i^2 + 2v_iv_0)^2 \\
&\geq & \sum_{i, j \in [k],\ i \neq j} (v_iv_j)^2 \\
&\geq& (k -1) \sum_{i \in [k]}  v_i^2,
\end{eqnarray*}
where the first inequality follows because we dropped a square term, and the second inequality follows from the fact that $v_i^2 \in \mathbb{Z}$ and $v_i \neq 0$ implies that $v_i^2 \geq 1$.

Therefore we obtain that $\|g\|_2 \geq \sqrt{(\|v\|_0 - 1)}\cdot \|v\|_2$ and the result follows.\qed
\end{proof}

\section{Computational experiments}

In this section, we describe our computational experiments, which are designed with two main purposes. Firstly, to test empirically if Eigen-CG inequalities beyond McCormick \eqref{eq:McCormick} and triangle \eqref{eq:triangle} can have a non-negligible impact on the dual bounds of non-convex QCQPs. And secondly, to verify if the depth bounds of Propositions \ref{proposition_depth} and \ref{proposition_depth-BH} are reflected empirically.
Since we have not yet found any Eigen-CG cuts that cannot be expressed as BH inequalities, we limit Eigen-CG cuts to BH inequalities in this section.

We remark that, in this work, we are not focused on efficiently computing cutting planes, but rather on understanding their expressive potential. This makes our approach complementary to previous work that has considered separation procedures for (subclasses of) BH inequalities. 
We refer the reader to Section \ref{sec:literature} for references to these previous works.
In our case, we aim to explore BH separation as exhaustively as possible at the expense of computational efficiency. As we explain below, we rely on full enumeration for generating sparse inequalities and on a nonconvex optimization problem for more general cases; in the latter, we employ heuristics to construct cuts within a reasonable time, but we do not impose a strict structure on them beforehand. Additionally, we strive to provide a systematic evaluation of how density affects the strength of BH inequalities, both in isolation and in conjunction with the standard SDP+McCormick relaxation—an aspect that, to the best of our knowledge, has not been examined in this way in prior work.

\subsection{Computational set-up}

\subsubsection{Instances}
We consider the BIQ instances from the Biq Mac library \cite{wiegele2007binary}. These are instances with a nonconvex quadratic objective function and binary variables.
We model these instances as non-convex QCQPs over continuous variables in the following way:
\begin{align*}
\min_{x} \quad & x^\T Q x \\
\text{s.t.} \quad & x_i(1 - x_i) = 0, \quad  i=1,\ldots, n \\
& x \in [0,1]^n
\end{align*}

We chose these instances since the traditional SDP relaxation provides a good dual bound, but it does not close the optimality gap. This gives us a trade-off that allows us to distinguish the effects of the different inequalities we consider more clearly.

We use the following 30 instances from the Biq Mac library \cite{wiegele2007binary} in our experiments:
$$\texttt{be100.1--be100.10}, \quad 
\texttt{be120.8.1--be120.8.10}, \quad 
\texttt{be150.8.1--be150.8.10}.
$$

\paragraph{Baseline relaxations}
For each instance, we consider the relaxation obtained from dropping the constraint $X - xx^\T = 0$ in \eqref{QCQP-Ext} and adding one of the following options:\\
\begin{enumerate}
\item[(i)] all McCormick inequalities;
\item[(ii)] all McCormick inequalities together with all triangle inequalities;
\item[(iii)] all McCormick inequalities and $X - xx^\T \succeq 0$; and
\item[(iv)] all McCormick inequalities, all triangle inequalities, and $X - xx^\T \succeq 0$.\\
\end{enumerate}

We will extract the dual bounds obtained by each of these settings, and evaluate how much they can be improved by adding {BH} cuts on top.
We note that, in the case of triangle inequalities, we will also consider variants given by only adding violated ones in a cutting plane fashion.
We describe these below.

Since we cannot solve the BIQ instances to optimality within a $4$-hour time limit, we report the following optimality gap (GAP) as

\[
\textrm{GAP} (\%)= | \textrm{UB} -\textrm{LB} | /|\textrm{LB}|\times 100,
\]
where $``\textrm{UB}" $ is the value of the best solution found by Gurobi, and $``\textrm{LB}"$ is the dual bound obtained by any of the relaxations we construct here.

\subsubsection{Separation routines} \label{sec:separation}

As mentioned earlier, in our experiments, we would like to test the performance of Eigen-CG{---BH, more specifically---}cuts beyond McCormick and triangle inequalities.
We do so with two different approaches.\\

\paragraph{Eigen-CG inequalities for $n\in \{4,5\}$} For these dimensions, we know that all the facets of $\mbox{BQP}_n$ are BH inequalities.
Since these are small-dimensional objects, we can simply enumerate all such facets and use them as a lookup table. This means that, whenever we want to separate $(\hat{x},\hat{X})$ with an Eigen-CG inequality of sparsity $n\in \{4,5\}$, we will simply enumerate all possible $n\times n$ submatrices of \(\left[ \begin{array}{cc}
1 & x^\T \\
x & X
\end{array} 
\right]\) and check the precomputed facets of $\mbox{BQP}_n$ for possible cuts.\\

\paragraph{Eigen-CG inequalities for $n \geq 6$}
To generate denser inequalities, we introduce a BH separator that enables systematic generation of BH cuts.
Specifically, at a solution $(\hat{x},\hat{X})$, the separator aims at solving the following non-convex integer quadratic program 

\begin{align}
\label{eq:bh_sep_ip}
v_{\mathrm{cut}}
=
\min_{w_0,w\in\mathbb{Z}^n} \, & w_0(w_0-1)
+ \sum_{i\in[n]}w_i (w_i + 2 w_0 -1) \hat{x}_i
+ \sum_{1\leq i < j \leq n} 2w_i w_j \hat{X}_{ij}\\
\text{s.t. }\, & L_i \le w_i \le U_i,\quad  \forall i\in[n] \nonumber
\end{align}
where $L_i, U_i$ are predetermined bounds. In our implementation, we set $L_i=-2, U_i=2$.
If $v_{\mathrm{cut}}<0$ (in our implementation, we use a violation tolerance of
$v_{\mathrm{cut}}<-0.01$), we obtain a violated BH inequality.\\

Motivated by Proposition~\ref{proposition_depth}, we control the sparsity of $w$ by introducing auxiliary variables $u$ satisfying
\begin{align}
\label{eq:l1_sparsity}
u_i \ge w_i,\, 
u_i \ge -w_i,u_i\leq \bar{u}_i,\, \forall i\in[n],\, 
\sum_{i\in[n]} u_i \le \hat{U},
\end{align}
In our experiments, we set $\bar{u}_i=2$ for each $i\in[n]$ and $\hat{U}=10$.\\

As can be expected, solving each problem \eqref{eq:bh_sep_ip} is computationally expensive.
And even if we are more concerned with expressiveness than with efficiency, separation times can be prohibitively large.
Therefore, to reduce the computational burden, we extract a smaller subset of indices from the current solution $(\hat{x}, \hat{X})$ using a greedy refinement procedure. 
Starting from the full matrix, rows and the corresponding columns are removed iteratively. Let $R \subseteq \{1,\dots,n\}$ denote the index set of the current matrix. At each iteration, for each $j \in R$, we compute the smallest eigenvalue of the principal submatrix indexed by $R \setminus \{j\}$. The index $j$ attaining the minimum of these values is removed. The procedure terminates when $|R| = \lfloor 0.15 n \rfloor$.
From these remaining indices, we randomly select $\floor{0.10n}$ indices to construct a submatrix, based on which \eqref{eq:bh_sep_ip} is solved. For each separation problem, we impose either a 10-second or a 30-second time limit, and gather all feasible solutions within the time limit; the detailed procedure is described in Algorithm~\ref{alg_bh_sequential}.

In our implementation, to generate a larger pool of BH cuts, we repeat this random selection procedure $100$ times; each time selecting $\floor{0.10n}$ indices and solving the corresponding $v_{\mathrm{cut}}$ problem.

{As a final remark, we note that for $n=6$ we can also enumerate all facets of $\mbox{BQP}_n$; however, we do not use this enumeration as a lookup table since not all facets are Eigen-CG. This implies that (1) using non-Eigen-CG inequalities does not align with our purpose of evaluating the effectiveness of Eigen-CG, and (2) filtering to only facets of $\mbox{BQP}_6$ that are Eigen-CG may not capture \emph{all} potential expressiveness of Eigen-CG.}

\subsubsection{Hardware and software} All experiments are run on a computer with an Intel(R) Xeon(R) Gold 6258R processor running at 2.7 GHz, with up to four threads used. All instances are executed in Python 3.11.4, with calls to the solvers Gurobi \cite{gurobi} (version 12.0.1 with default settings) or MOSEK \cite{mosek} (version 11.0.14 with default settings).
In particular, Gurobi is used to solve the BH separation problem \eqref{eq:bh_sep_ip}, while MOSEK is employed for the SDP relaxations with valid inequalities.

\subsection{Results I: baselines} 

Table~\ref{table_be_100_bench} reports the objective values for baseline relaxations {\it{(i)–(iv)}} for instances $\texttt{be100.1--be100.10}$, together with the optimality gap between relaxation {\it{(iv)}} and the best upper bound obtained by Gurobi. We observe that relaxation {\it{(iv)}} successfully closes the gap for instances \texttt{be100.2}, \texttt{be100.3}, \texttt{be100.4}, and \texttt{be100.6}.
In all these cases, the gap is closed only when the triangle inequalities are added.

\begin{table}[t]
\centering
\caption{Numerical Results of Relaxations {\it{(i)-(iv)}} on Instances $\texttt{be100.1--be100.10}$.}
\setlength{\tabcolsep}{0.8pt} 
\renewcommand{\arraystretch}{1} 
\label{table_be_100_bench}
\scriptsize
\begin{center}
\begin{tabular}{|c|r|r|r|r|rr|r|}
\hline
& \multicolumn{1}{c|}{}                         & \multicolumn{1}{c|}{}                                                                                    & \multicolumn{1}{c|}{}                                                                            & \multicolumn{1}{c|}{}                                                                                          & \multicolumn{2}{c|}{Gurobi}                                                     & \multicolumn{1}{c|}{}                             \\ \cline{6-7}
\multirow{-2}{*}{Instances} & \multicolumn{1}{c|}{\multirow{-2}{*}{(i) MC}} & \multicolumn{1}{c|}{\multirow{-2}{*}{\begin{tabular}[c]{@{}c@{}}(ii) MC\\ \& all triangle\end{tabular}}} & \multicolumn{1}{c|}{\multirow{-2}{*}{\begin{tabular}[c]{@{}c@{}}(iii) MC\\ \& SDP\end{tabular}}} & \multicolumn{1}{c|}{\multirow{-2}{*}{\begin{tabular}[c]{@{}c@{}}(iv) MC \& SDP\\ \& all triangle\end{tabular}}} & \multicolumn{1}{c|}{UB}                               & \multicolumn{1}{c|}{LB} & \multicolumn{1}{c|}{\multirow{-2}{*}{(iv)   Gap}} \\ \hline
be100.1                     & -31482.50                                     & -12715.33                                                                                                & -9892.69                                                                                         & -9769.21                                                                                                       & \multicolumn{1}{r|}{-9748.00}                         & -10848.67               & 0.22\%                                            \\ \hline
be100.2                     & -31463.50                                     & -12447.83                                                                                                & -8958.84                                                                                         & {\color[HTML]{000000} -8837.00}                                                                                & \multicolumn{1}{r|}{{\color[HTML]{000000} -8837.00}}  & -9656.83                & 0.00\%                                            \\ \hline
be100.3                     & -30959.50                                     & -12128.17                                                                                                & -8818.41                                                                                         & {\color[HTML]{000000} -8758.00}                                                                                & \multicolumn{1}{r|}{{\color[HTML]{000000} -8758.00}}  & -9685.11                & 0.00\%                                            \\ \hline
be100.4                     & -31760.00                                     & -13118.83                                                                                                & -10113.02                                                                                        & {\color[HTML]{000000} -10028.00}                                                                               & \multicolumn{1}{r|}{{\color[HTML]{000000} -10028.00}} & -10851.92               & 0.00\%                                            \\ \hline
be100.5                     & -31051.50                                     & -12146.50                                                                                                & -8288.34                                                                                         & {\color[HTML]{000000} -8076.50}                                                                                & \multicolumn{1}{r|}{{\color[HTML]{000000} -8028.00}}  & -9071.95                & 0.60\%                                            \\ \hline
be100.6                     & -31356.00                                     & -12471.50                                                                                                & -9135.21                                                                                         & {\color[HTML]{000000} -9045.00}                                                                                & \multicolumn{1}{r|}{{\color[HTML]{000000} -9045.00}}  & -10064.84               & 0.00\%                                            \\ \hline
be100.7                     & -32013.00                                     & -13005.17                                                                                                & -9560.91                                                                                         & -9425.07                                                                                                       & \multicolumn{1}{r|}{-9413.00}                         & -10438.23               & 0.13\%                                            \\ \hline
be100.8                     & -31937.00                                     & -13368.66                                                                                                & -9987.95                                                                                         & -9817.87                                                                                                       & \multicolumn{1}{r|}{-9800.00}                         & -11006.55               & 0.18\%                                            \\ \hline
be100.9                     & -30005.00                                     & -11111.33                                                                                                & -6960.06                                                                                         & -6773.50                                                                                                       & \multicolumn{1}{r|}{-6770.00}                         & -7867.90                & 0.05\%                                            \\ \hline
be100.10                    & -31104.50                                     & -12206.33                                                                                                & -8053.58                                                                                         & -7854.44                                                                                                       & \multicolumn{1}{r|}{-7790.00}                         & -8972.14                & 0.83\%                                            \\ \hline
\end{tabular}
\end{center}
\end{table}

Table \ref{table_be_120_bench} presents results for {\it{(i)–(iv)}} on instances $\texttt{be120.8.1--be120.8.10}$. In several instances, such as $\texttt{be120.8.4}$, $\texttt{be120.8.5}$, and $\texttt{be120.8.10}$, the gap is fully closed after triangle inequalities are added, whereas other cases still exhibit gaps between 0.35\% and 1.21\%.

\begin{table}[t]
\centering
\caption{Numerical Results of Relaxations {\it{(i)-(iv)}} on Instances $\texttt{be120.8.1--be120.8.10}$.}
\setlength{\tabcolsep}{0.8pt} 
\renewcommand{\arraystretch}{1} 
\label{table_be_120_bench}
\scriptsize
\begin{center}
\begin{tabular}{|c|r|r|r|r|rr|r|}
\hline
& \multicolumn{1}{c|}{}                         & \multicolumn{1}{c|}{}                                                                                    & \multicolumn{1}{c|}{}                                                                            & \multicolumn{1}{c|}{}                                                                                          & \multicolumn{2}{c|}{Gurobi}                                                     & \multicolumn{1}{c|}{}                             \\ \cline{6-7}
\multirow{-2}{*}{Instances} & \multicolumn{1}{c|}{\multirow{-2}{*}{(i) MC}} & \multicolumn{1}{c|}{\multirow{-2}{*}{\begin{tabular}[c]{@{}c@{}}(ii) MC\\ \& all triangle\end{tabular}}} & \multicolumn{1}{c|}{\multirow{-2}{*}{\begin{tabular}[c]{@{}c@{}}(iii) MC\\ \& SDP\end{tabular}}} & \multicolumn{1}{c|}{\multirow{-2}{*}{\begin{tabular}[c]{@{}c@{}}(iv) MC \& SDP\\ \& all triangle\end{tabular}}} & \multicolumn{1}{c|}{UB}                               & \multicolumn{1}{c|}{LB} & \multicolumn{1}{c|}{\multirow{-2}{*}{(iv)   Gap}} \\ 
\hline
be120.8.1                  & -35971.00                                & -14052.67                                         & -9582.27                                     & -9340.85                                                  & \multicolumn{1}{r|}{-9229.00}  & -10729.57               & 1.21\%                                             \\ \hline
be120.8.2                  & -35427.00                                & -14005.00                                         & -9893.23                                     & -9692.37                                                  & \multicolumn{1}{r|}{-9623.00}  & -10904.69               & 0.72\%                                             \\ \hline
be120.8.3                  & -36248.50                                & -14332.00                                         & -10231.09                                    & -10018.53                                                 & \multicolumn{1}{r|}{-9979.00}  & -11239.16               & 0.40\%                                             \\ \hline
be120.8.4                  & -36028.50                                & -14607.67                                         & -11056.25                                    & -10925.00                                                 & \multicolumn{1}{r|}{-10925.00} & -12024.28               & 0.00\%                                             \\ \hline
be120.8.5                  & -35218.00                                & -13814.83                                         & -10530.60                                    & -10405.00                                                 & \multicolumn{1}{r|}{-10405.00} & -11433.52               & 0.00\%                                             \\ \hline
be120.8.6                  & -35552.00                                & -13587.33                                         & -9206.51                                     & -8980.68                                                  & \multicolumn{1}{r|}{-8907.00}  & -10297.10               & 0.83\%                                             \\ \hline
be120.8.7                  & -37102.50                                & -15319.83                                         & -11644.57                                    & -11459.97                                                 & \multicolumn{1}{r|}{-11420.00} & -12774.37               & 0.35\%                                             \\ \hline
be120.8.8                  & -36053.00                                & -14522.50                                         & -10386.83                                    & -10142.88                                                 & \multicolumn{1}{r|}{-10045.00} & -11543.61               & 0.97\%                                             \\ \hline
be120.8.9                  & -35239.50                                & -13896.33                                         & -9794.11                                     & -9596.85                                                  & \multicolumn{1}{r|}{-9553.00}  & -10952.12               & 0.46\%                                             \\ \hline
be120.8.10                 & -35237.50                                & -13855.67                                         & -10052.89                                    & -9960.00                                                  & \multicolumn{1}{r|}{-9960.00}  & -11049.11               & 0.00\%                                             \\ \hline
\end{tabular}
\end{center}
\end{table}

A consistent pattern is observed for the larger $\texttt{be150.8.1--be150.8.10}$ instances. In Table \ref{table_be_150_bench} we observe that the baseline relaxations {\it{(i)–(iv)}} achieve moderate tightness, with remaining gaps ranging from 0.44\% to 1.25\%.

In all three Tables \ref{table_be_100_bench}, \ref{table_be_120_bench}, and \ref{table_be_150_bench}, we observe that incorporating McCormick and triangle inequalities on top of the SDP relaxation leads to noticeable improvement compared with the purely linear relaxations {\it{(i)–(ii)}}. 

\begin{table}[t]
\centering
\caption{Numerical Results of Relaxations {\it{(i)-(iv)}} on Instances $\texttt{be150.8.1--be150.8.10}$.}
\setlength{\tabcolsep}{0.8pt} 
\renewcommand{\arraystretch}{1} 
\label{table_be_150_bench}
\scriptsize
\begin{center}
\begin{tabular}{|c|r|r|r|r|rr|r|}
\hline
& \multicolumn{1}{c|}{}                         & \multicolumn{1}{c|}{}                                                                                    & \multicolumn{1}{c|}{}                                                                            & \multicolumn{1}{c|}{}                                                                                          & \multicolumn{2}{c|}{Gurobi}                                                     & \multicolumn{1}{c|}{}                             \\ \cline{6-7}
\multirow{-2}{*}{Instances} & \multicolumn{1}{c|}{\multirow{-2}{*}{(i) MC}} & \multicolumn{1}{c|}{\multirow{-2}{*}{\begin{tabular}[c]{@{}c@{}}(ii) MC\\ \& all triangle\end{tabular}}} & \multicolumn{1}{c|}{\multirow{-2}{*}{\begin{tabular}[c]{@{}c@{}}(iii) MC\\ \& SDP\end{tabular}}} & \multicolumn{1}{c|}{\multirow{-2}{*}{\begin{tabular}[c]{@{}c@{}}(iv) MC \& SDP\\ \& all triangle\end{tabular}}} & \multicolumn{1}{c|}{UB}                               & \multicolumn{1}{c|}{LB} & \multicolumn{1}{c|}{\multirow{-2}{*}{(iv)   Gap}} \\ 
\hline
be150.8.1                  & -56680.00                                & -21891.50                                         & -14101.40                                    & -13762.52                                                 & \multicolumn{1}{r|}{-13621.00} & -15496.35               & 1.04\%                                             \\ \hline
be150.8.2                  & -56860.00                                & -22373.00                                         & -14323.11                                    & -14018.49                                                 & \multicolumn{1}{r|}{-13857.00} & -15807.34               & 1.17\%                                             \\ \hline
be150.8.3                  & -57017.50                                & -22514.67                                         & -15351.14                                    & -15090.83                                                 & \multicolumn{1}{r|}{-14968.00} & -16996.74               & 0.82\%                                             \\ \hline
be150.8.4                  & -56683.00                                & -22244.50                                         & -14428.90                                    & -14163.61                                                 & \multicolumn{1}{r|}{-14102.00} & -15965.60               & 0.44\%                                             \\ \hline
be150.8.5                  & -55277.50                                & -21482.50                                         & -14625.55                                    & -14390.51                                                 & \multicolumn{1}{r|}{-14317.00} & -16302.01               & 0.51\%                                             \\ \hline
be150.8.6                  & -56987.00                                & -22526.67                                         & -15194.51                                    & -14834.58                                                 & \multicolumn{1}{r|}{-14651.00} & -16863.75               & 1.25\%                                             \\ \hline
be150.8.7                  & -57586.00                                & -23120.50                                         & -16162.70                                    & -15902.68                                                 & \multicolumn{1}{r|}{-15737.00} & -17766.95               & 1.05\%                                             \\ \hline
be150.8.8                  & -57053.00                                & -22858.67                                         & -15630.37                                    & -15376.82                                                 & \multicolumn{1}{r|}{-15268.00} & -17321.10               & 0.71\%                                             \\ \hline
be150.8.9                  & -55692.00                                & -21308.83                                         & -13560.56                                    & -13267.00                                                 & \multicolumn{1}{r|}{-13117.00} & -15144.38               & 1.14\%                                             \\ \hline
be150.8.10                 & -55953.50                                & -21974.83                                         & -14660.53                                    & -14389.32                                                 & \multicolumn{1}{r|}{-14232.00} & -16301.89               & 1.11\%                                             \\ \hline
\end{tabular}
\end{center}
\end{table}

\subsection{Results II: cuts on top of SDP via lookup tables}

When using the relaxation (iv), the number of constraints becomes very large.
While in this work, we are not focused on efficiency, to be able to simply test denser Eigen-CG inequalities, we need more manageable formulations.

To mitigate this computational burden, rather than adding all constraints simultaneously, we add only the violated ones iteratively in a cutting plane fashion.
Specifically, instead of {\it{(iv)}}, starting from setting {\it{(iii)}}, we first identify the \emph{violated} triangle inequalities and sort them in descending order of their violation magnitudes.
After resolving the model, we identify the violated $\mathrm{BQP}_4$ facets (via a lookup table, as described in Section \ref{sec:separation}) and sort them by their violation levels, using a tolerance of $10^{-3}$. We perform this separation procedure for $\mathrm{BQP}_4$ facets once.
The resulting model, consisting of all McCormick inequalities, the SDP relaxation, violated triangle inequalities, and violated $\mathrm{BQP}_4$ facets, is referred to as relaxation {\it{(v)}}. 

Building on relaxation {\it{(v)}}, we perform the same separation procedure for $\mathrm{BQP}_5$: using the precomputed enumeration of all $\mathrm{BQP}_5$ facets, we identify those that are violated (using a violation tolerance of $10^{-3}$), and sort them by their violation levels. We perform this separation procedure for $\mathrm{BQP}_5$ facets once. The resulting model, consisting of all McCormick inequalities, the SDP relaxation, violated triangle inequalities, and both violated $\mathrm{BQP}_4$ and $\mathrm{BQP}_5$ facets, is referred to as relaxation {\it{(vi)}}.

In Table~\ref{table_be_100_our_results}, we show the gaps for relaxations {\it{(v)}} and {\it{(vi)}} on instances whose gap was not closed by {\it{(iv)}}. We also report the running time and the number of additional violated cuts added.
The reported running time for relaxations {\it{(v)}} and {\it{(vi)}} includes both the time required to solve the model and the time spent identifying the violated facets.

\begin{table}[t]
\centering
\caption{Numerical Results of Relaxations {\it{(v)–(vi)}} on Instances \texttt{be100.1--be100.10}. Instances where the gap was closed by (iv) are excluded.}
\setlength{\tabcolsep}{0.8pt} 
\renewcommand{\arraystretch}{1} 
\label{table_be_100_our_results}
\scriptsize
\begin{center}
\begin{tabular}{|c|rrr|rrr|r|}
\hline
\multirow{2}{*}{Instances} 
& \multicolumn{3}{c|}{\makecell{(v) MC \& violated\\(triangle \& $\mathrm{BQP}_4$) \& SDP}} 
& \multicolumn{3}{c|}{\makecell{(vi) MC \& violated\\(triangle \& $\mathrm{BQP}_4$ \& $\mathrm{BQP}_5$) \& SDP}} 
& \multicolumn{1}{c|}{\multirow{2}{*}{(vi) Gap}} \\
\cline{2-7}
& \multicolumn{1}{c|}{Value}    & \multicolumn{1}{c|}{\# added $\mathrm{BQP}_4$ facets} & \multicolumn{1}{c|}{Time (s)} & \multicolumn{1}{c|}{Value}    & \multicolumn{1}{c|}{\# added $\mathrm{BQP}_5$ facets} & \multicolumn{1}{c|}{Time (s)} & \multicolumn{1}{c|}{}                           \\ \hline
be100.1  & \multicolumn{1}{r|}{-9750.28} & \multicolumn{1}{r|}{42894} & 269.17 
 & \multicolumn{1}{r|}{-9748.00} & \multicolumn{1}{r|}{87283} & 2603.72 & 0.00\% \\ \hline
be100.5  & \multicolumn{1}{r|}{-8039.07} & \multicolumn{1}{r|}{32064} & 228.86 
 & \multicolumn{1}{r|}{-8029.75} & \multicolumn{1}{r|}{55637} & 2567.30 & 0.02\% \\ \hline
be100.7  & \multicolumn{1}{r|}{-9413.00} & \multicolumn{1}{r|}{40519} & 245.66 
 & \multicolumn{3}{r|}{gap closed in (v)} & 0.00\% \\ \hline
be100.8  & \multicolumn{1}{r|}{-9800.00} & \multicolumn{1}{r|}{27289} & 203.68 
 & \multicolumn{3}{r|}{gap closed in (v)} & 0.00\% \\ \hline
be100.9  & \multicolumn{1}{r|}{-6770.00} & \multicolumn{1}{r|}{41151} & 220.52 
 & \multicolumn{3}{r|}{gap closed in (v)} & 0.00\% \\ \hline
be100.10 & \multicolumn{1}{r|}{-7808.60} & \multicolumn{1}{r|}{19721} & 202.39 
 & \multicolumn{1}{r|}{-7795.13} & \multicolumn{1}{r|}{46416} & 2517.58 & 0.07\% \\ \hline
\end{tabular}
\end{center}
\end{table}

Table~\ref{table_be_100_our_results} shows that relaxation {\it{(v)}} closes the gap for instances \texttt{be100.7}, \texttt{be100.8}, and \texttt{be100.9}, whereas relaxation {\it (vi)} closes the gap for instance \texttt{be100.1}. For the remaining two cases, \texttt{be100.5} and \texttt{be100.10}, the gaps are reduced from 0.60\% and 0.83\% (under relaxation {\it{(iv)}}) to 0.02\% and 0.07\%, respectively.
These results indicate that incorporating violated facet inequalities, particularly $\mathrm{BQP}_4$ and $\mathrm{BQP}_5$ facets, on top of the SDP relaxation, can substantially strengthen the formulation and improve relaxation quality.

The same reports for instances $\texttt{be120.8.1--be120.8.10}$ are summarized in Table \ref{table_be_120_our_results}.

\begin{table}[t]
\centering
\caption{Numerical Results of Relaxations {\it{(v)–(vi)}} on Instances \texttt{be120.8.1--be120.8.10}. Instances where the gap was closed by (iv) are excluded.}
\setlength{\tabcolsep}{0.8pt} 
\renewcommand{\arraystretch}{1} 
\label{table_be_120_our_results}
\scriptsize
\begin{center}
\begin{tabular}{|c|rrr|rrr|r|}
\hline
\multirow{2}{*}{Instances} 
& \multicolumn{3}{c|}{\makecell{(v) MC \& violated\\(triangle \& $\mathrm{BQP}_4$) \& SDP}} 
& \multicolumn{3}{c|}{\makecell{(vi) MC \& violated\\(triangle \& $\mathrm{BQP}_4$ \& $\mathrm{BQP}_5$) \& SDP}} 
& \multicolumn{1}{c|}{\multirow{2}{*}{(vi) Gap}} \\
\cline{2-7}
& \multicolumn{1}{c|}{Value}    & \multicolumn{1}{c|}{\# added $\mathrm{BQP}_4$ facets} & \multicolumn{1}{c|}{Time (s)} & \multicolumn{1}{c|}{Value}    & \multicolumn{1}{c|}{\# added $\mathrm{BQP}_5$ facets} & \multicolumn{1}{c|}{Time (s)} & \multicolumn{1}{c|}{}                           \\ \hline
be120.8.1  & \multicolumn{1}{r|}{-9277.81}  & \multicolumn{1}{r|}{27872}  & 422.99  & \multicolumn{1}{r|}{-9258.51}  & \multicolumn{1}{r|}{70405}  & 6063.66  & 0.32\% \\ \hline
be120.8.2  & \multicolumn{1}{r|}{-9638.72}  & \multicolumn{1}{r|}{30585}  & 449.04  & \multicolumn{1}{r|}{-9625.63}  & \multicolumn{1}{r|}{67741}  & 6027.92  & 0.03\% \\ \hline
be120.8.3  & \multicolumn{1}{r|}{-9981.26}  & \multicolumn{1}{r|}{49826}  & 452.79  & \multicolumn{1}{r|}{-9979.00}  & \multicolumn{1}{r|}{103365} & 6392.58  & 0.00\% \\ \hline
be120.8.6  & \multicolumn{1}{r|}{-8930.08}  & \multicolumn{1}{r|}{36184}  & 453.52 & \multicolumn{1}{r|}{-8914.52}  & \multicolumn{1}{r|}{73671}  & 6340.47  & 0.08\% \\ \hline
be120.8.7  & \multicolumn{1}{r|}{-11422.20} & \multicolumn{1}{r|}{43941}  & 488.66  & \multicolumn{1}{r|}{-11420.00} & \multicolumn{1}{r|}{85759}  & 6435.14  & 0.00\% \\ \hline
be120.8.8  & \multicolumn{1}{r|}{-10079.86} & \multicolumn{1}{r|}{27437}  & 439.79 & \multicolumn{1}{r|}{-10062.14} & \multicolumn{1}{r|}{68115}  & 6447.97  & 0.17\% \\ \hline
be120.8.9  & \multicolumn{1}{r|}{-9555.40}  & \multicolumn{1}{r|}{39785}  & 473.47  & \multicolumn{1}{r|}{-9553.00}  & \multicolumn{1}{r|}{77075}  & 6406.32  & 0.00\% \\ \hline
\end{tabular}
\end{center}
\end{table}

As shown in Table \ref{table_be_120_our_results}, the addition of violated facet cuts also tightens the relaxation significantly in these instances.
Relaxation {\it{(v)}} already provides significant improvement for most instances, and relaxation {\it{(vi)}} completes closing the gap for $\texttt{be120.8.2}$, $\texttt{be120.8.3}$, $\texttt{be120.8.7}$, and $\texttt{be120.8.9}$.
For the remaining cases, such as $\texttt{be120.8.1}$ and $\texttt{be120.8.8}$, the residual gaps are modest (0.32\% and 0.17\%, respectively).

As before, a consistent pattern is observed for the larger $\texttt{be150.8.1--be150.8.10}$ instances in Table \ref{table_be_150_our_results}.

\begin{table}[t]
\centering
\caption{Numerical Results of Relaxations {\it{(v)–(vi)}} on Instances \texttt{be150.8.1--be150.8.10}. Instances where the gap was closed by (iv) are excluded.}
\setlength{\tabcolsep}{0.8pt} 
\renewcommand{\arraystretch}{1} 
\label{table_be_150_our_results}
\scriptsize
\begin{center}
\begin{tabular}{|c|rrr|rrr|r|}
\hline
\multirow{2}{*}{Instances} 
& \multicolumn{3}{c|}{\makecell{(v) MC \& violated\\(triangle \& $\mathrm{BQP}_4$) \& SDP}} 
& \multicolumn{3}{c|}{\makecell{(vi) MC \& violated\\(triangle \& $\mathrm{BQP}_4$ \& $\mathrm{BQP}_5$) \& SDP}} 
& \multicolumn{1}{c|}{\multirow{2}{*}{(vi) Gap}} \\
\cline{2-7}
& \multicolumn{1}{c|}{Value}    & \multicolumn{1}{c|}{\# added $\mathrm{BQP}_4$ facets} & \multicolumn{1}{c|}{Time (s)} & \multicolumn{1}{c|}{Value}    & \multicolumn{1}{c|}{\# added $\mathrm{BQP}_5$ facets} & \multicolumn{1}{c|}{Time (s)} & \multicolumn{1}{c|}{}                           \\ \hline
be150.8.1  & \multicolumn{1}{r|}{-13681.81} & \multicolumn{1}{r|}{44261} & 1179.22 & \multicolumn{1}{r|}{-13653.03} & \multicolumn{1}{r|}{111883} & 19929.11 & 0.24\% \\ \hline
be150.8.2  & \multicolumn{1}{r|}{-13945.78} & \multicolumn{1}{r|}{44222} & 1173.56 & \multicolumn{1}{r|}{-13921.14} & \multicolumn{1}{r|}{112806} & 19784.94 & 0.46\% \\ \hline
be150.8.3  & \multicolumn{1}{r|}{-15032.16} & \multicolumn{1}{r|}{55058} & 1183.21 & \multicolumn{1}{r|}{-15007.50} & \multicolumn{1}{r|}{121648} & 19986.54 & 0.26\% \\ \hline
be150.8.4  & \multicolumn{1}{r|}{-14106.23} & \multicolumn{1}{r|}{53227} & 1267.24 & \multicolumn{1}{r|}{-14102.00} & \multicolumn{1}{r|}{115113} & 19887.74 & 0.00\% \\ \hline
be150.8.5  & \multicolumn{1}{r|}{-14338.13} & \multicolumn{1}{r|}{73947} & 1347.33 & \multicolumn{1}{r|}{-14321.23} & \multicolumn{1}{r|}{128066} & 19866.72 & 0.03\% \\ \hline
be150.8.6  & \multicolumn{1}{r|}{-14750.02} & \multicolumn{1}{r|}{43452} & 1184.93 & \multicolumn{1}{r|}{-14724.07} & \multicolumn{1}{r|}{113376} & 19752.62 & 0.50\% \\ \hline
be150.8.7  & \multicolumn{1}{r|}{-15836.11} & \multicolumn{1}{r|}{45535} & 1174.42 & \multicolumn{1}{r|}{-15814.57} & \multicolumn{1}{r|}{121053} & 19667.76 & 0.49\% \\ \hline
be150.8.8  & \multicolumn{1}{r|}{-15315.31} & \multicolumn{1}{r|}{56900} & 1277.34 & \multicolumn{1}{r|}{-15293.87} & \multicolumn{1}{r|}{116453} & 20877.77 & 0.17\% \\ \hline
be150.8.9  & \multicolumn{1}{r|}{-13195.24} & \multicolumn{1}{r|}{46391} & 1318.26 & \multicolumn{1}{r|}{-13171.32} & \multicolumn{1}{r|}{112957} & 20663.50 & 0.41\% \\ \hline
be150.8.10 & \multicolumn{1}{r|}{-14323.98} & \multicolumn{1}{r|}{47342} & 1253.24 & \multicolumn{1}{r|}{-14300.41} & \multicolumn{1}{r|}{115901} & 20665.29 & 0.48\% \\ \hline
\end{tabular}
\end{center}
\end{table}

As summarized in Table \ref{table_be_150_our_results}, relaxation {\it{(vi)}} closes the gap entirely for $\texttt{be150.8.4}$ and reduces the remaining gaps below 0.50\% for all other instances.
Although computational time increases with instance size, reaching about 20,000 seconds for the largest case, the overall strengthening effect of the added facet cuts remains evident.

In summary, across all benchmark sets ($\texttt{be100}$, $\texttt{be120.8}$, and $\texttt{be150.8}$), the additional facet-based cuts introduced in relaxations {\it{(v)–(vi)}} consistently enhance the quality of the SDP relaxation.
The improvements are achieved through a targeted, violation-based selection of $\mathrm{BQP}_4$ and $\mathrm{BQP}_5$ facets, leading to stronger bounds without introducing all possible constraints at once.
This demonstrates that these cuts are expressive and, since they are based on (relatively small) lookup tables, can provide a practical tool for further tightening the SDP relaxation.

As a final remark, we note that the marginal contribution of denser cuts (e.g., the ones derived from $\mathrm{BQP}_5$ versus $\mathrm{BQP}_4$) does degrade, as anticipated.
In the following results, we evaluate this behavior.

\subsection{Results III: cuts on top of SDP via BH separator}

To test the contribution of denser cuts, we initially built experiments that would incorporate the BH separator described in Section \ref{sec:separation} on top of relaxation {\it{(vi)}} above. In these experiments, we observed that almost no violated cut was found, and that the gap closed was almost exactly the same. For this reason, we decided not to report these results.
This is already a strong indication that, beyond $\mbox{BQP}_5$, the cuts we can generate would not be useful. However, to further test this hypothesis, we devised a new experiment.

Instead of building on top of relaxation {\it{(vi)}}, we take a step back and consider relaxation {\it{(iv)}}, which includes the McCormick inequalities,  all triangle inequalities, and the SDP constraint. From the solution $(\tilde x, \tilde X)$ we obtain, we apply the greedy refinement procedure of Section \ref{sec:separation} to identify a submatrix to use in the BH separator \eqref{eq:bh_sep_ip}. We repeat this process a total of $30$ times. The details of this procedure are summarized in Algorithm~\ref{alg_bh_sequential}.

\begin{algorithm}[t]
\caption{BH Cuts Generation Procedure}
\label{alg_bh_sequential}
\begin{algorithmic}[1]
\State \textbf{Step 1:} Start from the current solution $(\tilde x, \tilde X)$ obtained from relaxation {\it{(iv)}}
\State \textbf{Step 2:}  Use the greedy refinement procedure to find the submatrix with dimension $\floor{0.15n}$ 
\State \textbf{Step 3:} Randomly select $\floor{0.10n}$ indices from Step~2 and generate cuts using problem~\eqref{eq:bh_sep_ip}; repeat this sampling-and-separation step $100$ times
\State \textbf{Step 4:} Resolve the problem to update $(\tilde x, \tilde X)$, and repeat Steps 2–4 for up to $30$ iterations, or until no new BH cuts are generated in two consecutive iterations
\end{algorithmic}
\end{algorithm}

 For this relaxation, we consider two settings: for each \eqref{eq:bh_sep_ip} problem, either a 10-second or a 30-second time limit is imposed. 
 The numerical results are summarized in Table~\ref{table_be_BH_SDP}. From this table, we observe that the BH separator does improve the solution quality in several instances, but often at a noticeably higher computational cost, particularly under the 30-second setting. 
 Moreover, in contrast with relaxation {\it{(vi)}}, the improvement in objective value is generally modest, indicating that explicit facet enumeration of $\mathrm{BQP}_4$ and $\mathrm{BQP}_5$ plays a crucial role in strengthening the SDP relaxation.
 In this sense, Table~\ref{table_be_BH_SDP} illustrates that denser BH cuts can provide some enhancement, but they do not replace the effectiveness of relaxation {\it{(vi)}}; rather, they reinforce the conclusion that sparser Eigen-CG cuts are the most effective.

\begin{table}[t]
\centering
\caption{Numerical Results of the BH Cuts Generation Procedure on Instances \texttt{be100}, \texttt{be120.8}, and \texttt{be150.8}.} 
\setlength{\tabcolsep}{0.8pt} 
\renewcommand{\arraystretch}{1} 
\label{table_be_BH_SDP}
\scriptsize
\begin{center}
\begin{tabular}{|c|rrr|rrr|r|}
\hline
\multirow{2}{*}{Instances} 
& \multicolumn{3}{c|}{\begin{tabular}[c]{@{}c@{}}MC \&   SDP \& all triangle \\ \& BH separation (10-second limit)\end{tabular}} 
& \multicolumn{3}{c|}{\begin{tabular}[c]{@{}c@{}}MC \&   SDP \& all triangle \\ \& BH separation (30-second limit)\end{tabular}} 
& \multirow{2}{*}{Relaxation {\it{(vi)}}} \\ \cline{2-7}
& \multicolumn{1}{c|}{Value} 
& \multicolumn{1}{c|}{\# BH added} 
& \multicolumn{1}{c|}{Running time} 
& \multicolumn{1}{c|}{Value} 
& \multicolumn{1}{c|}{\# BH added} 
& \multicolumn{1}{c|}{Running time} 
& \\ \hline
be100.1    & \multicolumn{1}{r|}{-9766.93} & \multicolumn{1}{r|}{854} & 5784.95  & \multicolumn{1}{r|}{-9766.32} & \multicolumn{1}{r|}{797}  & 6218.35   & -9748.00 \\ \hline
be100.5    & \multicolumn{1}{r|}{-8075.49} & \multicolumn{1}{r|}{179} & 5475.48  & \multicolumn{1}{r|}{-8073.90} & \multicolumn{1}{r|}{819}  & 15319.46  & -8029.75 \\ \hline
be100.7    & \multicolumn{1}{r|}{-9423.65} & \multicolumn{1}{r|}{447} & 5073.24  & \multicolumn{1}{r|}{-9423.72} & \multicolumn{1}{r|}{501}  & 9546.32   & gap closed in {\it{(v)}} \\ \hline
be100.8    & \multicolumn{1}{r|}{-9816.39} & \multicolumn{1}{r|}{412} & 5646.64  & \multicolumn{1}{r|}{-9816.95} & \multicolumn{1}{r|}{290}  & 5809.17   & gap closed in {\it{(v)}} \\ \hline
be100.9    & \multicolumn{1}{r|}{-6772.91} & \multicolumn{1}{r|}{241} & 2197.46  & \multicolumn{1}{r|}{-6773.22} & \multicolumn{1}{r|}{112}  & 1562.94   & gap closed in {\it{(v)}} \\ \hline
be100.10   & \multicolumn{1}{r|}{-7853.21} & \multicolumn{1}{r|}{237} & 6445.39  & \multicolumn{1}{r|}{-7847.96} & \multicolumn{1}{r|}{1400} & 30650.69  & -7795.13 \\ \hline
be120.8.1  & \multicolumn{1}{r|}{-9337.17} & \multicolumn{1}{r|}{551} & 42899.11 & \multicolumn{1}{r|}{-9338.45} & \multicolumn{1}{r|}{759}  & 35608.74  & -9258.51 \\ \hline
be120.8.2  & \multicolumn{1}{r|}{-9687.68} & \multicolumn{1}{r|}{622} & 34676.93 & \multicolumn{1}{r|}{-9688.94} & \multicolumn{1}{r|}{1203} & 30578.74  & -9625.63 \\ \hline
be120.8.3  & \multicolumn{1}{r|}{-10018.38} & \multicolumn{1}{r|}{24}  & 4439.44  & \multicolumn{1}{r|}{-10016.43} & \multicolumn{1}{r|}{1031} & 46577.68  & -9979.00 \\ \hline
be120.8.6  & \multicolumn{1}{r|}{-8977.66} & \multicolumn{1}{r|}{310} & 22505.86 & \multicolumn{1}{r|}{-8974.25} & \multicolumn{1}{r|}{1695} & 82379.42  & -8914.52 \\ \hline
be120.8.7  & \multicolumn{1}{r|}{-11456.62} & \multicolumn{1}{r|}{437} & 24372.29 & \multicolumn{1}{r|}{-11454.78} & \multicolumn{1}{r|}{1849} & 85714.81  & -11420.00 \\ \hline
be120.8.8  & \multicolumn{1}{r|}{-10141.72} & \multicolumn{1}{r|}{140} & 12851.19 & \multicolumn{1}{r|}{-10142.86} & \multicolumn{1}{r|}{9}    & 4709.49   & -10062.14 \\ \hline
be120.8.9  & \multicolumn{1}{r|}{-9595.96} & \multicolumn{1}{r|}{305} & 6791.66  & \multicolumn{1}{r|}{-9590.60} & \multicolumn{1}{r|}{2299} & 74028.76  & -9553.00 \\ \hline
be150.8.1  & \multicolumn{1}{r|}{-13761.66} & \multicolumn{1}{r|}{60}  & 18354.76 & \multicolumn{1}{r|}{-13760.52} & \multicolumn{1}{r|}{279}  & 57996.73  & -13653.03 \\ \hline
be150.8.2  & \multicolumn{1}{r|}{-14018.49} & \multicolumn{1}{r|}{0}   & 5019.27  & \multicolumn{1}{r|}{-14016.14} & \multicolumn{1}{r|}{310}  & 76056.06  & -13921.14 \\ \hline
be150.8.3  & \multicolumn{1}{r|}{-15090.12} & \multicolumn{1}{r|}{24}  & 13703.38 & \multicolumn{1}{r|}{-15086.22} & \multicolumn{1}{r|}{776}  & 88043.01  & -15007.50 \\ \hline
be150.8.4  & \multicolumn{1}{r|}{-14163.51} & \multicolumn{1}{r|}{3}   & 7805.61  & \multicolumn{1}{r|}{-14162.19} & \multicolumn{1}{r|}{141}  & 68758.12  & -14102.00 \\ \hline
be150.8.5  & \multicolumn{1}{r|}{-14390.02} & \multicolumn{1}{r|}{33}  & 12683.10 & \multicolumn{1}{r|}{-14387.25} & \multicolumn{1}{r|}{495}  & 77851.09  & -14321.23 \\ \hline
be150.8.6  & \multicolumn{1}{r|}{-14834.24} & \multicolumn{1}{r|}{11}  & 10041.34 & \multicolumn{1}{r|}{-14832.02} & \multicolumn{1}{r|}{288}  & 48515.18  & -14724.07 \\ \hline
be150.8.7  & \multicolumn{1}{r|}{-15902.68} & \multicolumn{1}{r|}{0}   & 5092.76  & \multicolumn{1}{r|}{-15900.04} & \multicolumn{1}{r|}{264}  & 46351.17  & -15814.57 \\ \hline
be150.8.8  & \multicolumn{1}{r|}{-15376.76} & \multicolumn{1}{r|}{4}   & 7447.67  & \multicolumn{1}{r|}{-15375.60} & \multicolumn{1}{r|}{166}  & 27194.65  & -15293.87 \\ \hline
be150.8.9  & \multicolumn{1}{r|}{-13266.78} & \multicolumn{1}{r|}{24}  & 7673.11  & \multicolumn{1}{r|}{-13259.86} & \multicolumn{1}{r|}{958}  & 135163.28 & -13171.32 \\ \hline
be150.8.10 & \multicolumn{1}{r|}{-14389.15} & \multicolumn{1}{r|}{15}  & 10018.94 & \multicolumn{1}{r|}{-14388.78} & \multicolumn{1}{r|}{152}  & 18230.97  & -14300.41 \\ \hline
\end{tabular}
\end{center}
\end{table}

\subsection{Results IV: cuts added without the SDP relaxation}

As a final set of experiments, we examine the effect of strengthening the relaxation without incorporating the SDP constraints.
Since Propositions \ref{proposition_depth} and \ref{proposition_depth-BH} state that denser Eigen-CG cuts are weak on top of the SDP, we would like to test if, at least empirically, the same holds without the SDP constraint.

In these experiments, we start from a formulation that includes all McCormick inequalities. Then, we iteratively add violated triangle inequalities and violated facet inequalities from $\mathrm{BQP}_4$ and $\mathrm{BQP}_5$.
In the absence of SDP constraints, the number of violated facets can be substantial. Hence, identifying and sorting violated facets becomes computationally demanding and is limited by memory capacity. To mitigate this burden, we set the violation tolerance to $10^{-2}$. 
For violated $\mathrm{BQP}_4$ facets, we add at most 50,000 violated $\mathrm{BQP}_4$ cuts. For $\mathrm{BQP}_5$ cuts, rather than exhaustively enumerating all facets, we examine the $\binom{n}{5}$ index combinations, capped at 100,000 $\mathrm{BQP}_5$ index combinations, and identify the violated facets from this subset.
The combination of McCormick inequalities, violated triangles, and violated $\mathrm{BQP}_4$ facets is referred to as relaxation {\it{(vii)}}. Extending this by including the violated $\mathrm{BQP}_5$ facets yields relaxation {\it{(viii)}}.

To further enrich the cut pool beyond explicit facet enumeration, we incorporate the BH separation procedure with a 30-second time limit for each \eqref{eq:bh_sep_ip} problem; this enhanced approach is referred to as relaxation {\it{(ix)}}. 

The numerical results for instances \texttt{be100.1--be100.10} are shown in Table~\ref{table_be_100_without_sdp}.

\begin{table}[t]
\centering
\caption{Numerical Results of Relaxations {\it{(vii)–(ix)}} on Instances \texttt{be100.1--be100.10}.}
\setlength{\tabcolsep}{0.7pt} 
\renewcommand{\arraystretch}{1} 
\label{table_be_100_without_sdp}
\scriptsize
\begin{center}
\begin{tabular}{|c|rrr|rrr|rrr|}
\hline
\multirow{4}{*}{Instances} & \multicolumn{3}{c|}{\makecell{(vii) MC \& violated\\(triangle \& $\mathrm{BQP}_4$) }}           & \multicolumn{3}{c|}{\makecell{(viii) MC \& violated\\(triangle \& $\mathrm{BQP}_4$ \& $\mathrm{BQP}_5$) }}                                                   & \multicolumn{3}{c|}{\makecell{(ix) MC \& violated\\(triangle \& $\mathrm{BQP}_4$ \& $\mathrm{BQP}_5$)\\
\& BH separation (30-second limit)}}                                  \\ \cline{2-10} 
& \multicolumn{1}{c|}{Value}     & \multicolumn{1}{c|}{\makecell{added \\ $\mathrm{BQP}_4$\\ facets}} & \multicolumn{1}{c|}{Time (s)} & \multicolumn{1}{c|}{Value}     & \multicolumn{1}{c|}{\makecell{added \\ $\mathrm{BQP}_5$\\ facets}} & \multicolumn{1}{c|}{Time (s)} & \multicolumn{1}{c|}{Value}     & \multicolumn{1}{c|}{\makecell{added BH\\ ineqs}} & \multicolumn{1}{c|}{Time (s)} \\ \hline
be100.1                    & \multicolumn{1}{r|}{-11701.23} & \multicolumn{1}{r|}{50000}                         & 205.76                        & \multicolumn{1}{r|}{-11511.68} & \multicolumn{1}{r|}{56115}                         & 110.26                        & \multicolumn{1}{r|}{-10704.12} & \multicolumn{1}{r|}{33108}            & 18350.84                      \\ \hline
be100.2                    & \multicolumn{1}{r|}{-11389.77} & \multicolumn{1}{r|}{50000}                         & 196.07                        & \multicolumn{1}{r|}{-11176.09} & \multicolumn{1}{r|}{47992}                         & 98.76                         & \multicolumn{1}{r|}{-10290.18} & \multicolumn{1}{r|}{33793}            & 18476.56                      \\ \hline
be100.3                    & \multicolumn{1}{r|}{-10919.44} & \multicolumn{1}{r|}{50000}                         & 197.19                        & \multicolumn{1}{r|}{-10740.72} & \multicolumn{1}{r|}{40138}                         & 93.23                         & \multicolumn{1}{r|}{-9902.13}  & \multicolumn{1}{r|}{33687}            & 18324.92                      \\ \hline
be100.4                    & \multicolumn{1}{r|}{-12067.43} & \multicolumn{1}{r|}{50000}                         & 196.73                        & \multicolumn{1}{r|}{-11838.56} & \multicolumn{1}{r|}{42007}                         & 87.69                         & \multicolumn{1}{r|}{-11036.65} & \multicolumn{1}{r|}{33543}            & 18509.79                      \\ \hline
be100.5                    & \multicolumn{1}{r|}{-11070.20} & \multicolumn{1}{r|}{50000}                         & 197.39                        & \multicolumn{1}{r|}{-10884.21} & \multicolumn{1}{r|}{46405}                         & 100.48                        & \multicolumn{1}{r|}{-9978.32}  & \multicolumn{1}{r|}{33917}            & 18703.51                      \\ \hline
be100.6                    & \multicolumn{1}{r|}{-11348.03} & \multicolumn{1}{r|}{50000}                         & 194.88                        & \multicolumn{1}{r|}{-11140.01} & \multicolumn{1}{r|}{47770}                         & 100.79                        & \multicolumn{1}{r|}{-10248.21} & \multicolumn{1}{r|}{33430}            & 18286.14                      \\ \hline
be100.7                    & \multicolumn{1}{r|}{-11881.65} & \multicolumn{1}{r|}{50000}                         & 196.21                        & \multicolumn{1}{r|}{-11699.44} & \multicolumn{1}{r|}{78635}                         & 142.60                        & \multicolumn{1}{r|}{-10801.59} & \multicolumn{1}{r|}{33826}            & 19101.44                      \\ \hline
be100.8                    & \multicolumn{1}{r|}{-12223.01} & \multicolumn{1}{r|}{50000}                         & 201.36                        & \multicolumn{1}{r|}{-12039.60} & \multicolumn{1}{r|}{43846}                         & 94.17                         & \multicolumn{1}{r|}{-11172.92} & \multicolumn{1}{r|}{33645}            & 18878.65                      \\ \hline
be100.9                    & \multicolumn{1}{r|}{-9864.12}  & \multicolumn{1}{r|}{50000}                         & 196.27                        & \multicolumn{1}{r|}{-9693.58}  & \multicolumn{1}{r|}{60423}                         & 120.97                        & \multicolumn{1}{r|}{-8774.43}  & \multicolumn{1}{r|}{34057}            & 18253.72                      \\ \hline
be100.10                   & \multicolumn{1}{r|}{-11820.24} & \multicolumn{1}{r|}{50000}                         & 205.22                        & \multicolumn{1}{r|}{-11462.26} & \multicolumn{1}{r|}{276797}                        & 437.75                        & \multicolumn{1}{r|}{-10303.32} & \multicolumn{1}{r|}{34643}            & 24632.07                      \\ \hline
\end{tabular}
\end{center}
\end{table}

Although the computational cost of relaxation {\it{(ix)}} is significantly higher due to the repeated solution of \eqref{eq:bh_sep_ip}, these results indicate that denser BH cuts provide meaningful additional tightening over relaxations {\it{(vii)–(viii)}}.
This is considerably different from the experiments with the SDP constraint.

The results also illustrate that, in the absence of SDP relaxation, the number of violated $\mbox{BQP}_4$ and $\mbox{BQP}_5$ increases rapidly with $n$.
For this reason, we present numerical results only for instances \texttt{be100.1--be100.10}; in higher-dimensional cases, numerical instability may arise, making reliable implementation of the BH procedure significantly more challenging.

Finally, we note that, even though denser BH inequalities gain a more predominant role in the absence of an SDP constraint, the overall quality of the relaxations is considerably lower than that of the SDP-based ones.

\section{Conclusions}

In this article, we introduced Eigen-CG inequalities, obtained by applying Chvátal–Gomory rounding to eigenvector inequalities for QCQPs, and studied their structure and limitations.
These inequalities are closely related to known inequalities in the literature, but to the best of our knowledge, an in-depth analysis like this one has not been conducted.

We identified several nested subclasses of Eigen-CG cuts with strict inclusions, but we showed that their conic closures coincide with the closure of BH inequalities. This indicates, from a new angle, how expressive BH inequalities already are.
We conjecture that Eigen-CG inequalities have the same conic closure as the BH inequalities, but we have not yet found a proof.
We also showed that Eigen-CG inequalities do not capture all facets of the Boolean Quadric Polytope in dimension six, making explicit the limits of this construction.

Additionally, we studied the effect of density on Eigen-CG cuts and showed that it is detrimental in the presence of an SDP relaxation. Specifically, we proved that the depth of Eigen-CG (and BH) inequalities deteriorates quickly as their support grows when the SDP and McCormick constraints are enforced. 

Finally, we present computational experiments to test these observations empirically.
We observe that adding sparse facet inequalities derived from low-dimensional BQP to SDP+McCormick relaxations yields significant improvements in dual bounds. In contrast, denser BH cuts provide only a limited additional benefit at a much higher computational cost. 

Overall, our results suggest that progress in strengthening SDP relaxations for QCQPs is more likely to come from identifying and exploiting sparse, well-structured inequalities rather than from pursuing increasingly dense eigenvector-based cuts.

\bibliographystyle{plain}
\bibliography{eigen_cg_refs.bib}

\end{document}